\documentclass[12pt]{amsart}

\usepackage{tikz-cd}
\usepackage{amsfonts}
\usepackage{amsthm}
\usepackage{xfrac}
\usepackage{faktor}
\usepackage{amssymb}
\usepackage{enumitem}
\usepackage{microtype}
\usepackage{amsthm}
\usepackage{amsmath}
\usepackage{amssymb}
\usepackage{mathtools}
\usepackage{tikz-cd}
\usepackage{microtype}
\usepackage{calrsfs}
\DeclareMathAlphabet{\pazocal}{OMS}{zplm}{m}{n}
\usepackage{mathabx}
\usepackage{url}

\usepackage{amsmath}
\usepackage[all]{xy}
\usepackage{hyperref}

\title{ ON STRONGLY FP-INJECTIVE ACCESSIBILITY AND WEAKLY FP-PROJECTIVE PERIODICITY}
\author{Ilias Kaperonis and Konstantinos Golfis}

\thanks{}

\newtheorem{Lemma}{Lemma}[section]
\newtheorem{Proposition}[Lemma]{Proposition}
\newtheorem{Theorem}[Lemma]{Theorem}
\newtheorem{Corollary}[Lemma]{Corollary}
\newtheorem{Remark}[Lemma]{Remark}
\newtheorem{Definition}[Lemma]{Definition}
\newtheorem{Example}[Lemma]{Example}
\newtheorem*{theoremstar}{Theorem}

\begin{document}

\begin{abstract} 
Accessibility results are established for (Cartan-Eilenberg) strongly FP-injective complexes, by first proving the accessibility of Strongly FP-injective modules and then utilizing the pseudopullback theorem. We also prove that every acyclic complex of weakly FP-projective modules has weakly FP-projective cycles over left $<\aleph_m$-coherent rings for a non-negative integer $m$. This is achieved by introducing the cotorsion pairs $\bigl(\mathcal{C}\aleph_mP\text{-}proj,\mathcal{C}\aleph_mP\text{-}inj\bigr)$, cogenerated by all $\aleph_m$-presentable weakly FP-projective modules, and proving that acyclic complexes of modules from the class $\mathcal{C}\aleph_mP\text{-}proj$ have weakly FP-projective cycles for every non-negative integer $m$.

\end{abstract}

\maketitle
\tableofcontents

\section{Introduction}

In homological algebra, duality plays an important role. Flat modules seem to have two natural dual counterparts, namely,  absolutely pure modules (also called FP-injective modules) and strongly FP-injective modules. Flat modules and absolutely pure modules have been studied for many years, but recently a lot of work has also been done on strongly FP-injective modules (see, for example, \cite{Bazzoni2}, \cite{EK}, \cite{ET} and \cite{LGO}). Flat modules can be defined by declaring a module $F$ to be flat if every epimorphism with codomain $F$ is pure. Dualizing this definition one gets absolutely pure modules, meaning that a module $A$ is absolutely pure if and only if every monomorphism with domain $A$ is pure. Another defining property of flat modules is that flat modules always have pure acyclic projective resolutions. Dualizing this definition we get the class of strongly FP-injective modules, meaning that a module $S$ is strongly FP-injective  if and only if every injective resolution of $S$ is a pure acyclic complex. We point out that all left strongly FP-injective modules are left absolutely pure modules and that the two concepts coincide if and only if the underlying ring is left coherent.

Of course, due to the definition, flat modules share dual (or analogous) properties with both absolutely pure and strongly FP-injective modules. For example, the class $\mathsf{Flat}$ of all flat modules is the left hand side of the cotorsion pair $\bigl(\mathsf{Flat},\mathsf{Cot}\bigr)$ where $\mathsf{{Cot}}$ is the class of all cotorsion modules. The classes $\mathsf{fpInj}$ of all absolutely pure modules and $\mathsf{SfpInj}$ of all strongly FP-injective modules are the right hand sides of the cotorsion pairs $\bigl(\mathsf{fpProj},\mathsf{fpInj}\bigr)$ and $\bigl(\mathsf{WfpProj},\mathsf{SfpInj}\bigr)$ respectively. Modules in the class $\mathsf{fpProj}$ are called FP-projective modules and modules in the class $\mathsf{WfpProj}$ are called weakly FP-projective modules. All three cotorsion pairs are complete while only $\bigl(\mathsf{Flat},\mathsf{Cot}\bigr)$ and $\bigl(\mathsf{WfpProj},\mathsf{SfpInj}\bigr)$ are hereditary in general.

The failure of the hereditary condition for $\bigl(\mathsf{fpProj},\mathsf{fpInj}\bigr)$ suggests that strongly FP-injective modules share more dual properties with flat modules than absolutely pure modules. One instance of this, is a, in his own words, puzzling result of \v{S}\v{t}ov\'{\i}\v{c}ek, namely, every acyclic complex of injective modules has cotorsion cycles (actually every such cycle is right orthogonal to all modules of finite flat dimension, see \cite[Corollary 5.9]{Stovicekcoderived}). The dual result is proved in \cite[Corollary 4.9(ii)]{EK}: The cycles of an acyclic complex of projective modules are left orthogonal to all modules of finite strongly FP-injective dimension. Bazzoni, Cortés-Izurdiaga and Estrada generalized the already mentioned result of \v{S}\v{t}ov\'{\i}\v{c}ek in \cite{Bazzoni}, proving that every cycle of an acyclic complex of cotorsion modules is itself cotorsion. A natural question arises: 
\begin{center}
    Is every cycle of an acyclic complex of weakly FP-projective modules itself weakly FP-projective?
\end{center} 
We give a positive answer for $<\aleph_m$-coherent rings when $m$ is a non-negative integer. This appears as Corollary \ref{Maincorollary} in the main text:
\begin{theoremstar}Let $R$ be a $<\aleph_m$-coherent ring where $m$ is a non-negative integer. Then acyclic complexes of weakly FP-projective modules have weakly FP-projective cycles.
    
\end{theoremstar}

In this paper, we also study strongly FP-injective complexes and Cartan-Eilenberg strongly FP-injective complexes. There are two associated cotorsion pairs, namely $\bigl(\operatorname{dg}\mathsf{WfpProj}(R),\widetilde{\mathsf{SfpInj}}(R)\bigr)$ in the abelian exact structure and $\bigl(\operatorname{CE}(\mathsf{WfpProj}(R)),\operatorname{CE}(\mathsf{SfpInj}(R))\bigr)$ in the Cartan-Eilenberg exact structure. While the completeness of these pairs is known, we give a common generating set for both cases, proving functioriality of the completeness. The complexes in the class $\widetilde{\mathsf{SfpInj}}(R)$ are the ones called strongly FP-injective complexes whereas the complexes in the class $\operatorname{CE}(\mathsf{SfpInj}(R))$ are called Cartan-Eilenberg strongly FP-injective complexes. 

The paper is organized as follows: We begin with some preliminary notions required and then Section 3 is focused on the study of (Cartan-Eilenberg) strongly FP-injective complexes, with a small subsection exploring the dg-exact structure and dg-strongly FP-injective complexes. Sections 4 and 5 establish the classes $\widetilde{\mathsf{SfpInj}}(R)$ and $\operatorname{CE}(\mathsf{SfpInj}(R))$ as accessible categories, respectively, using the pseudopullback theorem. Section 6 deals with strongly FP-injective modules, characterizing $<\lambda$-coherent rings for regular cardinals $\lambda$ in terms of $\lambda$-directed colimits of strongly FP-injective modules. Finally, section 7 establishes the Theorem mentioned above. We note that over left $<\aleph_m$ coherent rings, for non-negative integers $m$, this means that acyclic complexes of weakly FP-projective modules belong to the classes $\operatorname{dg}\mathsf{WfpProj}(R)$ and $\operatorname{CE}(\mathsf{WfpProj}(R))$.

\section{Preliminaries} \noindent{\em Modules.} 
All modules considered in this paper will be left modules over an associative unital ring $R$. We denote by $R$-$\textrm{Mod}$ the category of all $R$-modules.

\vspace{0.1in}

\noindent{\em Complexes.} A complex 
$$(C,d) =\cdots\rightarrow  C_{n+1}  \xrightarrow{d_{n+1}} C_{n} \xrightarrow{d_{n}} C_{n-1}\rightarrow \cdots $$ will be denoted by $C$ instead of $(C,d)$. An $R$-module $M$ is viewed as the complex 
$$S^0(M)=\cdots\rightarrow  0  \rightarrow M \rightarrow  0\rightarrow \cdots $$
where $M$ is in the 0th degree. Similarly, we denote by $S^n(M)$ the complex with $M$ as its nth entry and zeros elsewhere. In addition, we denote by $D^n(M)$ the complex  
$$D^n(M)=\cdots\rightarrow  0  \rightarrow M \xrightarrow{id}  M\rightarrow 0\rightarrow \cdots $$
where M is in degrees n and (n+1).

Given a complex $C$ we let $Z(C)$ and $B(C)$ be the complexes of cycles and boundaries, respectively, so $H(C)=Z(C)/B(C)$ denotes the homology complex of $C$. We also denote by $\Sigma(C)$ the suspension of $C$. So $\Sigma(C)_n = C_{n-1}$ and $d_{\Sigma(C)} = -d_{C}$, $\Sigma^k(C)$ is defined  by $\Sigma^k(C)_n=C_{n-k}$ and $d_{\Sigma^k(C)} = (-1)^kd_{C}$ for all $k\in\mathsf{Z}$.

For two complexes $C$ and $D$ we denote by $\operatorname{Hom}_{Ch(R)}(C,D)$ the abelian group of chain maps from $C$ to $D$ and by $\operatorname{Ext}^{i}(C,D)$ the right derived functors of $\operatorname{Hom}_{Ch(R)}$. The Hom-complex of $C,D$ is denoted by $\operatorname{Hom}_R(C,D)$.  Following \cite{GR}, for two complexes $C, D$, we define  $\underline{\operatorname{Hom}}(C,D) = Z(\operatorname{Hom}_R(C,D))$, $\underline{\operatorname{Hom}}(C,D)$ can be made into a complex with $\underline{\operatorname{Hom}}(C,D)_m$ being
the abelian group of morphisms from $C$ to $\Sigma^{-m}(D)$ and with a boundary operator given by: 
if $f \in \underline{\operatorname{Hom}}(C,D)_m$, then $d_m(f) \colon C \to \Sigma^{-(m-1)}(D)$ with 
$d_m(f)_n = (-1)^m d^D \circ f_n, \forall n \in \mathbb{Z}$. We also set for a complex $C$, $C^+=\underline{\operatorname{Hom}}(C,D^0(\mathbb{Q} / \mathbb{Z}))$.

\begin{Definition} Let $\mathcal{X}$ be a class of modules. By $CE(\mathcal{X})$ we denote the class of all complexes $C$ such that $C,Z(C),B(C)$ and $H(C)$ are complexes of modules belonging to $\mathcal{X}$.
    
\end{Definition}

To declare that a complex $C$ belongs to the class $CE(\mathcal{X})$ we will say that $C$ is a $CE$-$\mathcal{X}$ complex instead of writing $C\in CE(\mathcal{X})$. For example, if $Proj$ denotes the class of all projective modules, the complexes in $CE(Proj)$ are called $CE$-projective complexes. The same applies to the classes $Inj$ and $Flat$ of injective and flat modules.

\begin{Proposition}(\cite[Section 6]{En}) The functor $\operatorname{Hom}(-,-)$ on $Ch(R)\times Ch(R)$ is right balanced by $CE(Proj)\times CE(Inj)$.
    
\end{Proposition}

This result says that we can compute derived functors of
$\operatorname{Hom}(-,-)$ using either $CE$ projective resolutions or $CE$ injective resolutions. For given $C$ and $D$, we will denote these derived functors applied to $\operatorname{Hom}(C,D)$ as $\overline{\operatorname{Ext}}^n(C,D)$.

    

\section{Exact structures on the category of complexes}

Strong FP-injectivity gives rise to several classes of complexes,
depending on the exact structure we are working in. In the ordinary
abelian exact structure, strongly FP-injective complexes are acyclic
and have strongly FP-injective cycles. In the Cartan-Eilenberg exact
structure, nonzero homology is allowed, and the corresponding condition
is expressed in terms of components and cycles.

We first establish these characterizations and describe the associated
cotorsion pairs and approximation sequences. A common set of sphere
and disk complexes will make the difference between ordinary and
Cartan-Eilenberg orthogonality explicit. We then turn to the exact
structure induced by dg-projective and dg-injective complexes. Although
this exact category need not be efficient, the ordinary induced
cotorsion pairs still provide complete hereditary cotorsion pairs in
this structure.

We recall a few notions on exact categories needed below. 
A cotorsion pair $(\mathcal L,\mathcal R)$ in an exact category is a pair satisfying $\mathcal L={}^{\perp_1}\mathcal R$ and
$\mathcal R=\mathcal L^{\perp_1}$, with Ext computed in that exact
structure. It is complete if every object $X$ admits conflations
$0\to R_X\to L_X\to X\to0$ and
$0\to X\to R^X\to L^X\to0$, with the indicated terms in
$\mathcal L$ and $\mathcal R$. These give special precovers and special
preenvelopes, respectively. It is functorially complete if the two
sequences can be chosen functorially in $X$, and hereditary if all
positive Ext groups between its left and right classes vanish.
A cotorsion pair is cogenerated by a set $\mathcal S$ if its right
class is $\mathcal S^{\perp_1}$. Specifying such a set gives explicit
Ext tests for membership in the right class. Under the hypotheses of
\cite[Theorem 5.16]{StovicekExact}, it also yields functorial
completeness and a description of the left class as retracts of
$\mathcal S$-filtered objects.

  \subsection{Strongly FP-injective complexes in the abelian exact structure}

A module $M$ is strongly FP-injective if $\operatorname{Ext}^i_R(F,M)=0$
for every finitely presented module $F$ and every $i\geq1$; equivalently,
all the cosyzygies in an injective resolution of $M$, including $M$ itself,
are FP-injective \cite{LGO}. We denote the class of strongly
FP-injective modules by $\mathsf{SfpInj}(R)$ and the class of weakly
FP-projective modules by
\[
\mathsf{WfpProj}(R)={}^{\perp_1}\mathsf{SfpInj}(R).
\]
We begin with the ordinary abelian exact structure. Throughout this subsection, Ext is computed in the ordinary abelian
category of modules or complexes, as appropriate.

\begin{Definition}
A complex is finitely presented if it is bounded and all its components
are finitely presented modules. A complex $C$ is FP-injective if
$\operatorname{Ext}^1_{\operatorname{Ch}(R)}(F,C)=0$ for every finitely
presented complex $F$. It is strongly FP-injective if $\operatorname{Ext}^i_{\operatorname{Ch}(R)}(F,C)=0$ for every finitely presented complex $F$ and every $i\geq1$.
\end{Definition}

\begin{Theorem}\label{Theor}
For a complex $C$, the following conditions are equivalent:
\begin{enumerate}[label=(\roman*)]
\item $C$ is strongly FP-injective.
\item $C_n\in\mathsf{SfpInj}(R)$ for every $n$, and $\operatorname{Hom}_R(F,C)$ is acyclic for
every finitely presented complex $F$.
\item $C$ is a pure acyclic complex of strongly FP-injective modules.
\end{enumerate}
\end{Theorem}
\begin{proof}
Recall that an FP-injective complex is precisely a pure acyclic complex
of FP-injective modules; equivalently, it is an acyclic complex with
FP-injective cycles \cite[Theorem 2.10]{WL}. Suppose (i) holds. Every cosyzygy of an injective resolution of $C$ in
$\operatorname{Ch}(R)$ is FP-injective, by dimension shifting. Injective
objects of $\operatorname{Ch}(R)$ are contractible complexes with injective
cycles, and their components are injective modules. Evaluating the
resolution in degree $n$ therefore gives an injective resolution of $C_n$
whose cosyzygies are FP-injective. Thus $C_n\in\mathsf{SfpInj}(R)$. The asserted
Hom-acyclicity follows from the FP-injectivity of $C$, proving (ii).

The implication (ii)$\Rightarrow$(iii) is given by the equivalence of (i) and (v) in {\rm(\cite[Proposition 2.8]{E})} and the fact that any chain map from any finitely presented complex $F$ to $C$ being null-homotopic means exactly that $\operatorname{Hom}_R(F,C)$ is exact for any finitely presented complex $F$.

For (iii)$\Rightarrow$(i), let 
\[0\rightarrow C \rightarrow I^0 \rightarrow I^1\rightarrow\cdots\] 
be an injective resolution of $C$, and consider the short exact
sequences
\[
0\longrightarrow J^k\longrightarrow I^k\longrightarrow J^{k+1}
\longrightarrow0\qquad(k\geq0),\qquad J^0=C.
\]
The class $\mathsf{SfpInj}(R)$ is closed under cokernels of monomorphisms and thus, inductively, every $J^k$ has components in $\mathsf{SfpInj}(R)$. Since these components are
FP-injective, each displayed sequence is degreewise pure. The complexes
$I^k$ are contractible. Tensoring the displayed sequences degreewise with
an arbitrary right module, and using the homology long exact sequence,
shows inductively that every $J^k$ is pure acyclic. Hence every $J^k$ is
FP-injective as an object of $\operatorname{Ch}(R)$. Dimension shifting
now proves (i).
\end{proof}
For the approximation results below, it is useful to express the
preceding characterization entirely in terms of cycles. In this spirit, we have the following; compare with \cite[Theorem 2.10]{WL}.
\begin{Corollary}\label{sfi-cycles}
A complex $C$ is strongly FP-injective if and only if it is acyclic and
$Z_nC\in\mathsf{SfpInj}(R)$ for every $n\in\mathbb Z$.
\end{Corollary}
\begin{proof}
If $C$ is strongly FP-injective, Theorem~\ref{Theor} shows that its
components belong to $\mathsf{SfpInj}(R)$ and that its cycles are FP-injective.
For a finitely presented module $F$, the sequences
$0\to Z_nC\to C_n\to Z_{n-1}C\to0$ give
\[
\operatorname{Ext}^i_R(F,Z_nC)
\cong\operatorname{Ext}^{i-1}_R(F,Z_{n-1}C)\qquad(i\geq2).
\]
Iterating reduces to a vanishing first Ext group, so $Z_nC\in\mathsf{SfpInj}(R)$.
Conversely, if the cycles belong to $\mathsf{SfpInj}(R)$, extension closure gives
$C_n\in\mathsf{SfpInj}(R)$. The same short exact sequences are pure because
$Z_nC$ is FP-injective. Thus $C$ is pure acyclic, and
Theorem~\ref{Theor} applies.
\end{proof}

The distinction from FP-injectivity is already visible on disk
complexes. If $R$ is not left coherent, there is an FP-injective module
$M$ which is not strongly FP-injective \cite[Theorem 4.2]{LGO}. Then
$D^0(M)$ has the same distinction in $\operatorname{Ch}(R)$.
The next theorem shows that left coherence is exactly the condition
under which the two classes of complexes coincide, and records the
corresponding colimit and character-duality criteria.

\begin{Theorem}\label{theorem1}
The following conditions are equivalent:
\begin{enumerate}[label=(\roman*)]
\item $R$ is left coherent.
\item FP-injective complexes and strongly FP-injective complexes coincide.
\item Every directed colimit of strongly FP-injective complexes is
FP-injective.
\item For every complex $C$ of left modules, $C$ is strongly FP-injective
if and only if $C^+$ is a flat complex of right modules.
\item For every complex $C$ of left modules, $C$ is strongly FP-injective
if and only if $C^{++}$ is strongly FP-injective.
\end{enumerate}
Under these equivalent conditions, the colimits in (iii) are strongly
FP-injective.
\end{Theorem}
\begin{proof}
Over a left coherent ring, FP-injective modules and strongly FP-injective
modules coincide by \cite[Theorem 4.2]{LGO}. Thus (i)$\Rightarrow$(ii)
follows from Corollary~\ref{sfi-cycles} and
\cite[Theorem 2.10]{WL}. 

Let $M$ be an FP-injective module. Then $D^0(M)$ is an FP-injective complex, so by (ii) it is a strongly FP-injective complex, meaning that $M$ is a strongly FP-injective module and therefore $R$ is left coherent. Hence, we get (ii)$\Rightarrow$(i).

Under (i), directed colimits of modules in $\mathsf{SfpInj}(R)$ remain in
$\mathsf{SfpInj}(R)$. Directed colimits are exact in $R$-Mod and commute with
cycles. Corollary~\ref{sfi-cycles} therefore proves (iii), with the
stronger conclusion stated. Conversely, disks commute with directed
colimits, so (iii) implies the module-theoretic colimit criterion for
left coherence in \cite[Theorem 4.2]{LGO}.

Under (i), condition (ii) and \cite[Corollary 2.14]{WL} give (iv) and
(v). For the reverse implications, test (iv) or (v) on $D^0(M)$ for
every left module $M$. The character complex of a disk is a disk of
$M^+=\operatorname{Hom}_{\mathbb Z}(M,\mathbb Q/\mathbb Z)$, up to
the indexing shift, and its double character complex is a disk of
$M^{++}$. A disk is flat, respectively strongly FP-injective, precisely
when its defining module is so. We obtain respectively the criteria
$M\in\mathsf{SfpInj}(R)\Longleftrightarrow M^+\text{ is flat}$ and
$M\in\mathsf{SfpInj}(R)\Longleftrightarrow M^{++}\in\mathsf{SfpInj}(R)$ for every $M$.
Either criterion forces every FP-injective module to belong to
$\mathsf{SfpInj}(R)$. Indeed, let $M$ be FP-injective and embed it purely in
an injective module $I$. Dualizing a pure exact sequence splits it, so
$M^+$ is a direct summand of $I^+$ and $M^{++}$ is a direct summand
of $I^{++}$. Under the first criterion, $I^+$ is flat, hence so is
$M^+$, and then $M\in\mathsf{SfpInj}(R)$. Under the second, $I^{++}$ belongs
to $\mathsf{SfpInj}(R)$, hence so does $M^{++}$, and again $M\in\mathsf{SfpInj}(R)$.
The equality of FP-injectives and strongly FP-injectives implies left
coherence again by \cite[Theorem 4.2]{LGO}.
\end{proof}

We now turn to approximation sequences. The starting point is the cotorsion pair formed by weakly FP-projective and strongly FP-injective
modules. We record its cogenerating set explicitly, since its spheres
and disks will be used in both the ordinary and the Cartan-Eilenberg
exact structures.

The following result is established in \cite[Theorem 3.4]{LGO}, but we include a proof for the convenience of the reader.
\begin{Proposition}\label{module-pair}
Let $\mathcal F$ be a set of representatives of the finitely presented
left modules, including $R$. Choose a projective resolution of each
$F\in\mathcal F$ and put
\[
\mathcal T=\{\Omega^jF:F\in\mathcal F,\ j\geq0\},\qquad\Omega^0F=F.
\]
Then $\mathsf{SfpInj}(R)=\mathcal T^{\perp_1}$, and
$(\mathsf{WfpProj}(R),\mathsf{SfpInj}(R))$ is a functorially complete
hereditary cotorsion pair.
\end{Proposition}
\begin{proof}
For each fixed projective resolution, dimension shifting gives
\[
\operatorname{Ext}^1_R(\Omega^jF,M)
\cong\operatorname{Ext}^{j+1}_R(F,M)\qquad(j\geq0).
\]
Thus $M\in\mathcal T^{\perp_1}$ if and only if all positive Ext groups
from finitely presented modules to $M$ vanish, which proves
$\mathsf{SfpInj}(R)=\mathcal T^{\perp_1}$. Since
$\mathsf{WfpProj}(R)={}^{\perp_1}\mathsf{SfpInj}(R)$, we have
$\mathcal T\subseteq\mathsf{WfpProj}(R)$ and hence
\[
\mathsf{SfpInj}(R)\subseteq\mathsf{WfpProj}(R)^{\perp_1}
\subseteq\mathcal T^{\perp_1}=\mathsf{SfpInj}(R).
\]
This verifies both orthogonal equalities. The module category is
efficient with its abelian exact structure, and
$R\in\mathcal T\subseteq\operatorname{Filt}(\mathcal T)$ is a
generator. Therefore \cite[Proposition 3.13 and Theorem 5.16]{StovicekExact}
apply and give functorial completeness.

The long exact Ext sequence shows that $\mathsf{SfpInj}(R)$ contains the
injectives and is closed under extensions and cokernels of
monomorphisms between its objects. Consequently, all cosyzygies
$B^j$ in an injective resolution of $B\in\mathsf{SfpInj}(R)$ again belong
to $\mathsf{SfpInj}(R)$. For $A\in\mathsf{WfpProj}(R)$ and $i\geq1$, dimension
shifting gives
$\operatorname{Ext}^i_R(A,B)\cong
\operatorname{Ext}^1_R(A,B^{i-1})=0$, proving hereditarity. 
\end{proof}

For a class of modules $\mathcal L$, write
\[
\widetilde{\mathcal L}=
\{X:X\text{ is acyclic and }Z_nX\in\mathcal L\text{ for all }n\}.
\]
Thus $\widetilde{\mathsf{SfpInj}}(R)$ is precisely the class of strongly
FP-injective complexes. For the pair of Proposition~\ref{module-pair},
define
\begin{align*}
\operatorname{dg}\mathsf{WfpProj}(R)
&=\{P:P_n\in\mathsf{WfpProj}(R)\ \forall n,\ \operatorname{Hom}_R(P,B)\text{ is acyclic}
       \ \forall B\in\widetilde{\mathsf{SfpInj}}(R)\},\\
\operatorname{dg}\mathsf{SfpInj}(R)
&=\{I:I_n\in\mathsf{SfpInj}(R)\ \forall n,\ \operatorname{Hom}_R(A,I)\text{ is acyclic}
       \ \forall A\in\widetilde{\mathsf{WfpProj}}(R)\}.
\end{align*}
These are the induced classes of \cite[Definition 3.3]{GillespieFlat}.
The Hom-complex conditions are equivalent to null-homotopy of every
chain map in question, since both test classes are closed under shifts.

\begin{Theorem}\label{induced-pairs}
For any ring $R$, the pairs
\[
(\operatorname{dg}\mathsf{WfpProj}(R),\widetilde{\mathsf{SfpInj}}(R)),\qquad
(\widetilde{\mathsf{WfpProj}}(R),\operatorname{dg}\mathsf{SfpInj}(R))
\]
are complete hereditary cotorsion pairs in the abelian category
$\operatorname{Ch}(R)$. If $\mathsf{Ex}$ is the class of acyclic
complexes, then
\[
\operatorname{dg}\mathsf{WfpProj}(R)\cap\mathsf{Ex}=\widetilde{\mathsf{WfpProj}}(R),
\qquad
\operatorname{dg}\mathsf{SfpInj}(R)\cap\mathsf{Ex}=\widetilde{\mathsf{SfpInj}}(R).
\]
\end{Theorem}
\begin{proof}
Proposition~\ref{module-pair} supplies a complete hereditary cotorsion
pair in $R$-Mod. Applying \cite[Corollaries 3.8 and 3.13]{GillespieFlat}
gives the two hereditary cotorsion pairs and the displayed compatibility
equalities. Their completeness follows from
\cite[Theorem 3.5]{YangLiu}.
\end{proof}

The preceding theorem already gives the existence of special
approximations for both pairs. We now make the first pair explicit:
a set of sphere and disk complexes will cogenerate it. This also
allows the approximations for this pair to be chosen functorially
and describes $\operatorname{dg}\mathsf{WfpProj}(R)$ by filtrations. We begin
with the Ext computations needed to identify the right orthogonal.

\begin{Lemma}\label{ordinary-tests}
For every module $M$, complex $X$, and integer $n$,
\[
\operatorname{Ext}^i_{\operatorname{Ch}(R)}(D^n(M),X)
\cong\operatorname{Ext}^i_R(M,X_{n+1})\quad(i\geq1),
\]
and $\operatorname{Ext}^1_{\operatorname{Ch}(R)}(S^n(R),X)
\cong H_{n-1}(X)$. If $X$ is acyclic, then also
\[
\operatorname{Ext}^i_{\operatorname{Ch}(R)}(S^n(M),X)
\cong\operatorname{Ext}^i_R(M,Z_nX)\quad(i\geq1).
\]
\end{Lemma}
\begin{proof}
Apply $D^n$ to a projective module resolution of $M$ and use the
adjunction
\[
\operatorname{Hom}_{\operatorname{Ch}(R)}(D^n(M),X)
\cong\operatorname{Hom}_R(M,X_{n+1}).
\]
For the second assertion, apply
$\operatorname{Hom}_{\operatorname{Ch}(R)}(-,X)$ to
$0\to S^{n-1}(R)\to D^{n-1}(R)\to S^n(R)\to0$.
Since $D^{n-1}(R)$ is projective in $\operatorname{Ch}(R)$, the
resulting exact sequence contains
\[
X_n\xrightarrow{d_n}Z_{n-1}X\longrightarrow
\operatorname{Ext}^1_{\operatorname{Ch}(R)}(S^n(R),X)
\longrightarrow0.
\]
Its cokernel is $H_{n-1}(X)$, as required.

Finally, take an injective resolution of an acyclic $X$ in
$\operatorname{Ch}(R)$. All its cosyzygies are acyclic, since the
injective complexes are acyclic. Taking $n$th cycles in each short exact
sequence of the resolution is exact: a cycle in the quotient lifts to
a cycle because the kernel is acyclic. Since the cycles of an injective
complex are injective modules, this gives an injective resolution of
$Z_nX$. The identity
$\operatorname{Hom}_{\operatorname{Ch}(R)}(S^n(M),-)
=\operatorname{Hom}_R(M,Z_n(-))$ proves the last assertion.
\end{proof}

\begin{Proposition}\label{ordinary-set}
With $\mathcal T$ as in Proposition~\ref{module-pair}, set
\[
\mathcal U=\{S^n(T),D^n(T):T\in\mathcal T,\ n\in\mathbb Z\},
\qquad
G=\bigoplus_{n\in\mathbb Z}\bigl(S^n(R)\oplus D^n(R)\bigr),
\qquad\mathcal U'=\mathcal U\cup\{G\}.
\]
Then
\[
\mathcal U^{\perp_1}=\mathcal U'^{\perp_1}
=\widetilde{\mathsf{SfpInj}}(R)
\quad\text{in }\operatorname{Ch}(R).
\]
The first pair in Theorem~\ref{induced-pairs} is functorially complete,
and
\[
\operatorname{dg}\mathsf{WfpProj}(R)
=\operatorname{Retr}\operatorname{Filt}_{\mathrm{ab}}(\mathcal U').
\]
Here $\operatorname{Filt}_{\mathrm{ab}}$ denotes continuous transfinite
filtrations starting at zero with successive quotients in $\mathcal U'$
and ordinary short exact successor sequences; $\operatorname{Retr}$
denotes closure under direct summands.
\end{Proposition}
\begin{proof}
Suppose $X\in\mathcal U^{\perp_1}$. Since $R\in\mathcal T$,
Lemma~\ref{ordinary-tests} gives
\[
H_{n-1}(X)\cong
\operatorname{Ext}^1_{\operatorname{Ch}(R)}(S^n(R),X)=0
\qquad(n\in\mathbb Z),
\]
so $X$ is acyclic. The same lemma now identifies the remaining sphere
tests with
$\operatorname{Ext}^1_R(T,Z_nX)=0$ for all $T\in\mathcal T$.
Thus $Z_nX\in\mathcal T^{\perp_1}=\mathsf{SfpInj}(R)$ and
$X\in\widetilde{\mathsf{SfpInj}}(R)$. Conversely, if
$X\in\widetilde{\mathsf{SfpInj}}(R)$, the exact sequences
$0\to Z_nX\to X_n\to Z_{n-1}X\to0$ and extension closure of
$\mathsf{SfpInj}(R)$ imply $X_n\in\mathsf{SfpInj}(R)$. Both the sphere and disk
Ext groups therefore vanish by Lemma~\ref{ordinary-tests}.

Coproducts of projective resolutions are projective resolutions in
$\operatorname{Ch}(R)$. Applying Hom into $X$ converts such a coproduct
into a product of cochain complexes of abelian groups. Products of
abelian groups are exact, so taking cohomology yields
\[
\operatorname{Ext}^1_{\operatorname{Ch}(R)}(G,X)
\cong\prod_{n\in\mathbb Z}
\left(\operatorname{Ext}^1_{\operatorname{Ch}(R)}(S^n(R),X)
\times\operatorname{Ext}^1_{\operatorname{Ch}(R)}(D^n(R),X)\right).
\]
It follows that adjoining $G$ does not change the right orthogonal.
The disk summands generate $\operatorname{Ch}(R)$: free epimorphisms
onto the components of any complex induce an epimorphism from a
coproduct of disks of $R$. Consequently, $G$ is a generator and
belongs to $\operatorname{Filt}_{\mathrm{ab}}(\mathcal U')$.

The category $\operatorname{Ch}(R)$ is Grothendieck, hence efficient
by \cite[Proposition 3.13]{StovicekExact}. We can therefore apply
\cite[Theorem 5.16]{StovicekExact} to the set $\mathcal U'$.
It gives functorial completeness and identifies
${}^{\perp_1}(\mathcal U'^{\perp_1})$ with the retracts of
$\mathcal U'$-filtered objects. By Theorem~\ref{induced-pairs}, this
left orthogonal is $\operatorname{dg}\mathsf{WfpProj}(R)$.
\end{proof}
The resulting approximations have two different uses. They give a
special strongly FP-injective preenvelope of every complex, and they
also provide quasi-isomorphic replacements by dg-weakly FP-projective
and dg-strongly FP-injective complexes. These are recorded separately
in the following sequences.
\begin{Corollary}\label{ordinary-approximations}
Every complex $X$ admits short exact sequences
\begin{align*}
&0\longrightarrow B_X\longrightarrow P_X\longrightarrow X\longrightarrow0,
&&B_X\in\widetilde{\mathsf{SfpInj}}(R),\quad P_X\in\operatorname{dg}\mathsf{WfpProj}(R),\\
&0\longrightarrow X\longrightarrow B^X\longrightarrow P^X\longrightarrow0,
&&B^X\in\widetilde{\mathsf{SfpInj}}(R),\quad P^X\in\operatorname{dg}\mathsf{WfpProj}(R),\\
&0\longrightarrow X\longrightarrow I_X\longrightarrow A_X\longrightarrow0,
&&I_X\in\operatorname{dg}\mathsf{SfpInj}(R),\quad A_X\in\widetilde{\mathsf{WfpProj}}(R).
\end{align*}
The first two can be chosen functorially. The map $X\to B^X$ is a
special strongly FP-injective preenvelope. The maps $P_X\to X$ and
$X\to I_X$ are quasi-isomorphisms.
\end{Corollary}
\begin{proof}
The sequences follow from Theorem~\ref{induced-pairs} and
Proposition~\ref{ordinary-set}. Applying Hom and using the first Ext
vanishing gives the precover and preenvelope factorization properties.
The quasi-isomorphism assertions follow from the acyclic kernel $B_X$
and acyclic cokernel $A_X$, respectively.
\end{proof}

\begin{Corollary}\label{sfi-model}
There is a hereditary abelian model structure on $\operatorname{Ch}(R)$
with cofibrant objects $\operatorname{dg}\mathsf{WfpProj}(R)$, fibrant objects
$\operatorname{dg}\mathsf{SfpInj}(R)$, and weak equivalences the
quasi-isomorphisms. Its trivially fibrant objects are exactly the
strongly FP-injective complexes, and its homotopy category is $D(R)$.
\end{Corollary}
\begin{proof}
Apply \cite[Corollary 3.8]{YangLiu} to Proposition~\ref{module-pair}.
The two compatible complete cotorsion pairs are those of
Theorem~\ref{induced-pairs}; compatibility identifies the trivially
fibrant objects with $\widetilde{\mathsf{SfpInj}}(R)$. Localization at the
quasi-isomorphisms is the ordinary unbounded derived category $D(R)$.
\end{proof}

\subsection{Cartan-Eilenberg Strongly FP-injective complexes and the Cartan-Eilenberg exact structure}

The ordinary strongly FP-injective complexes are necessarily acyclic.
To obtain approximations which also act on cycles, boundaries and
homology, we will use the Cartan-Eilenberg exact structure. A short exact sequence of complexes is CE-exact if it is exact degreewise and also on cycles, boundaries, cokernels and homology. This is equivalent to being degreewise exact and exact on cycles \cite[Lemma 5.2]{En}. These sequences form the Cartan-Eilenberg exact structure
$\mathcal E_{\mathrm{CE}}$ on $\operatorname{Ch}(R)$.

This exact category has enough projectives and injectives, namely
$\operatorname{CE}(Proj)$ and $\operatorname{CE}(Inj)$
\cite[Theorems 5.5 and 5.7]{En} and is efficient
\cite[Proposition 3.5]{EIW}. Consequently, its Yoneda Ext agrees
with the relative derived functor introduced in Section 2:
\[
\operatorname{Ext}^i_{\mathrm{CE}}(X,Y)
=\overline{\operatorname{Ext}}^{\,i}(X,Y)\qquad(i\geq1).
\]
The results of \cite[Theorem 9.4]{En} and
\cite[Theorem 3.9]{YangLiangCE} yield a complete hereditary
CE cotorsion pair from the module pair in
Proposition~\ref{module-pair}. We will refine this conclusion by identifying
an explicit cogenerating set, proving functorial completeness, and
describing $\operatorname{CE}(\mathsf{WfpProj}(R))$ by CE-exact filtrations.

\begin{Definition}
A complex is CE finitely presented if it is bounded and its components,
cycles, boundaries and homology modules are finitely presented.
A complex $C$ is CE-strongly FP-injective if $\operatorname{Ext}^i_{\mathrm{CE}}(F,C)=0$ for every CE finitely presented $F$ and every $i\geq1$. The corresponding $\operatorname{Ext_{CE}^1}$ condition defines a CE-FP-injective complex, cf. \cite{LL}.
\end{Definition}
The sphere and disk computations now detect cycles and components
without an acyclicity assumption on the target. This is the basic
difference from the ordinary Ext calculations of
Lemma~\ref{ordinary-tests}.
\begin{Lemma}\label{EnExtiso}
For a module $M$, a complex $X$, and integers $n$ and $i\geq1$, there
are natural isomorphisms
\[
\operatorname{Ext}^i_{\mathrm{CE}}(S^n(M),X)
\cong\operatorname{Ext}^i_R(M,Z_nX),\qquad
\operatorname{Ext}^i_{\mathrm{CE}}(D^n(M),X)
\cong\operatorname{Ext}^i_R(M,X_{n+1}).
\]
\end{Lemma}
\begin{proof}
Apply $S^n$, respectively $D^n$, to a projective resolution of $M$.
The resulting augmented sequences are CE-exact and their terms are
CE-projective, so they are CE-projective resolutions. The ordinary Hom
identities with $Z_nX$ and $X_{n+1}$ identify the cochain complexes
computing the two sides.
\end{proof}
These computations extend the module orthogonality conditions to CE
classes. We state the argument for an arbitrary class of modules and
then specialize to the finitely presented modules.
\begin{Lemma}\label{lemma}
Let $\mathcal C$ be a class of modules containing zero, and put
$\mathcal D=\mathcal C^{\perp_{\geq1}}$. Then
\[
\operatorname{CE}(\mathcal D)
=\operatorname{CE}(\mathcal C)^{\perp_{\geq1,\mathrm{CE}}}
=\bigl(\operatorname{CE}(\mathcal C)^b\bigr)^{\perp_{\geq1,\mathrm{CE}}},
\]
where the superscript $b$ denotes bounded complexes.
\end{Lemma}
\begin{proof}
The class $\mathcal D$ contains the injectives and is closed under
extensions and cokernels of monomorphisms. Set $\mathcal L={}^{\perp_1}\mathcal D$.
Every syzygy of a module in $\mathcal C$ belongs to $\mathcal L$,
by dimension shifting. It follows that
$\mathcal L^{\perp_1}=\mathcal D$: one inclusion follows from the
definition of $\mathcal L$, and the other by testing against these
syzygies. Thus $(\mathcal L,\mathcal D)$ is a cotorsion pair.
Its right class is coresolving, so it is hereditary. 

Since $\mathcal C\subseteq\mathcal L$, we have
$\operatorname{CE}(\mathcal C)\subseteq\operatorname{CE}(\mathcal L)$.
Applying \cite[Theorem 9.4]{En} to the hereditary pair
$(\mathcal L,\mathcal D)$ therefore gives
\[
\operatorname{Ext}^i_{\mathrm{CE}}(F,X)=0
\quad(F\in\operatorname{CE}(\mathcal C),\
X\in\operatorname{CE}(\mathcal D),\ i\geq1).
\]
Conversely, suppose $X$ is right orthogonal in all positive degrees to
the bounded CE-$\mathcal C$ complexes. These include $S^n(M)$ and
$D^n(M)$ for all $M\in\mathcal C$. Lemma~\ref{EnExtiso} gives
$X_n,Z_nX\in\mathcal D$. The short exact sequences
\[
0\to Z_nX\to X_n\to B_{n-1}X\to0,
\qquad 0\to B_nX\to Z_nX\to H_nX\to0
\]
and the cokernel closure of $\mathcal D$ give $B_nX,H_nX\in\mathcal D$.
Thus $X\in\operatorname{CE}(\mathcal D)$, proving both equalities.
\end{proof}

\begin{Corollary}\label{Cor1}
The CE-strongly FP-injective complexes are precisely
$\operatorname{CE}(\mathsf{SfpInj}(R))$.
\end{Corollary}
\begin{proof}
Apply Lemma~\ref{lemma} with $\mathcal C$ the class of finitely
presented modules. Then $\mathcal D=\mathsf{SfpInj}(R)$, and the bounded
CE-$\mathcal C$ complexes are precisely the CE finitely presented ones.
\end{proof}

\begin{Proposition}\label{prop1}
A complex $X$ is CE-strongly FP-injective if and only if
$X_n,Z_nX\in\mathsf{SfpInj}(R)$ for every $n$. In particular, an acyclic
complex is CE-strongly FP-injective if and only if it is strongly
FP-injective.
\end{Proposition}
\begin{proof}
The two short exact sequences in the proof of Lemma~\ref{lemma} and
the cokernel closure of $\mathsf{SfpInj}(R)$ show that the component and cycle
conditions imply the boundary and homology conditions. Now apply
Corollary~\ref{Cor1}. For an acyclic complex, use also
Corollary~\ref{sfi-cycles}.
\end{proof}

\begin{Remark}\label{ce-closure}
The class $\operatorname{CE}(\mathsf{SfpInj}(R))$ is closed under products,
direct summands, CE-extensions and cokernels of CE-inflations between
its objects. Indeed, $\mathsf{SfpInj}(R)$ is closed under products and direct
summands and is coresolving; products commute with cycles, and the CE
structure makes the component and cycle sequences exact.
\end{Remark}

Unlike ordinary strongly FP-injective complexes, these complexes may
have nonzero homology. For example, if $M$ is a nonzero strongly
FP-injective module, then $S^0(M)$ is CE-strongly FP-injective but is
not strongly FP-injective in the ordinary sense. As in the ordinary
case, left coherence determines the distinction between FP-injectivity
and its strong version, as the next theorem shows.
\begin{Theorem}\label{ce-coherence}
The following conditions are equivalent:
\begin{enumerate}[label=(\roman*)]
\item $R$ is left coherent.
\item CE-strongly FP-injective complexes and CE-FP-injective complexes coincide.
\item Every directed colimit of CE-strongly FP-injective complexes is
CE-strongly FP-injective.
\item Every directed colimit of CE-injective complexes is CE-FP-injective.
\item For every complex $C$ of left modules, $C$ is CE-strongly
FP-injective if and only if $C^+$ is CE-flat.
\item For every complex $C$ of left modules, $C$ is CE-strongly
FP-injective if and only if $C^{++}$ is CE-strongly FP-injective.
\end{enumerate}
\end{Theorem}
\begin{proof}
Under (i), the module classes of FP-injectives and strongly
FP-injectives coincide. Proposition~\ref{prop1} and
\cite[Theorem 4.4]{LL} prove (ii). Conversely, applying (ii) to stalks
gives equality of the two module classes, hence (i) by
\cite[Theorem 4.2]{LGO}; stalks of FP-injective modules are CE-FP-injective
by \cite[Example 4.3]{LL}.

Under (i), $\mathsf{SfpInj}(R)$ is closed under directed colimits, and cycles
commute with these colimits. Proposition~\ref{prop1} gives (iii).
Condition (iii) implies (iv), since CE-injectives belong to
$\operatorname{CE}(\mathsf{SfpInj}(R))$ and CE-strongly FP-injective complexes
are CE-FP-injective. Applying (iv) to directed systems of stalks of
injective modules gives the module colimit criterion for coherence
used in \cite[Theorem 4.2]{LGO}, so (iv) implies (i).

Under (i), condition (ii) and \cite[Theorem 5.1]{LL} give (v) and (vi).
Conversely, apply (v) or (vi) to every stalk $S^0(M)$. Character duals
of stalks are stalks of character modules, up to shift. CE-flatness of
a stalk is flatness of its module, and CE-strong FP-injectivity of a
stalk is membership in $\mathsf{SfpInj}(R)$. We obtain the corresponding
module criteria appearing in the proof of Theorem~\ref{theorem1}.
The pure-embedding argument given there shows that either criterion
forces all FP-injective modules to be strongly FP-injective, hence (i).
\end{proof}
We now return to approximations. Having identified the CE-strongly
FP-injective complexes, we can compute the CE right orthogonal of the
same sphere and disk set used in the ordinary case.
\begin{Theorem}\label{ce-complete}
For any ring $R$, the pair
\[
\bigl(\operatorname{CE}(\mathsf{WfpProj}(R)),\operatorname{CE}(\mathsf{SfpInj}(R))\bigr)
\]
is a functorially complete hereditary cotorsion pair
in $(\operatorname{Ch}(R),\mathcal E_{\mathrm{CE}})$. More precisely,
\[
\mathcal U^{\perp_{1,\mathrm{CE}}}
=\mathcal U'^{\perp_{1,\mathrm{CE}}}
=\operatorname{CE}(\mathsf{SfpInj}(R)),\qquad
\operatorname{CE}(\mathsf{WfpProj}(R))
=\operatorname{Retr}\operatorname{Filt}_{\mathrm{CE}}(\mathcal U').
\]
The filtrations in the last expression are continuous transfinite
filtrations whose successor sequences are CE-exact and whose successive
quotients belong to $\mathcal U'$.
\end{Theorem}
\begin{proof}
Proposition~\ref{module-pair} and \cite[Theorem 9.4]{En} give the
hereditary cotorsion pair. By Lemma~\ref{EnExtiso}, orthogonality to
$\mathcal U$ means exactly
\[
\operatorname{Ext}^1_R(T,Z_nX)=0
=\operatorname{Ext}^1_R(T,X_{n+1})
\quad(T\in\mathcal T,\ n\in\mathbb Z).
\]
Since $\mathcal T^{\perp_1}=\mathsf{SfpInj}(R)$, Proposition~\ref{prop1}
identifies this orthogonal with $\operatorname{CE}(\mathsf{SfpInj}(R))$.

The object $G$ of Proposition~\ref{ordinary-set} is CE-projective:
each of its sphere and disk summands is CE-projective, and coproducts
of projective objects remain projective. We also verify that $G$ is
a generator for this exact structure. Choose free epimorphisms $F_n\to X_n$ and $Q_n\to Z_nX$ for every $n$.
By the disk and sphere adjunctions, they define a chain map
\[
\bigoplus_{n\in\mathbb Z}
\bigl(D^{n-1}(F_n)\oplus S^n(Q_n)\bigr)\longrightarrow X.
\]
The disk maps make it surjective in every degree, and the sphere maps
make it surjective on cycles. Its kernel sequence is therefore
CE-exact. The domain is a coproduct of spheres and disks of $R$,
hence is a direct summand of a sufficiently large coproduct of copies
of $G$. Adding the complementary summand, mapped to zero, gives a
CE-deflation from that coproduct itself.

Adjoining the CE-projective object $G$ to $\mathcal U$ does not change
the relative orthogonal. Hence $\operatorname{Filt}_{\mathrm{CE}}
(\mathcal U')$ contains the generator $G$. Together with the
efficiency recalled at the beginning of the section, this verifies
the hypotheses of \cite[Theorem 5.16]{StovicekExact}. Its conclusion
applies to the pair
\[
\bigl({}^{\perp_{1,\mathrm{CE}}}
(\mathcal U'^{\perp_{1,\mathrm{CE}}}),
\mathcal U'^{\perp_{1,\mathrm{CE}}}\bigr).
\]
The right class is $\operatorname{CE}(\mathsf{SfpInj}(R))$ by the computation
above, and the left class is $\operatorname{CE}(\mathsf{WfpProj}(R))$ by the
cotorsion-pair equality established at the start of the proof.
We obtain functorial completeness and the claimed description by
retracts of CE-filtered objects.
\end{proof}
CE-exactness makes these approximations compatible with all four
module families attached to a complex. In particular, the induced
maps on homology are themselves special module approximations.
\begin{Corollary}\label{ce-approximations}
Every complex $X$ admits functorially chosen CE-exact sequences
\[
0\longrightarrow V_X\longrightarrow U_X\longrightarrow X\longrightarrow0,
\qquad
0\longrightarrow X\longrightarrow V^X\longrightarrow U^X\longrightarrow0,
\]
where $U_X,U^X\in\operatorname{CE}(\mathsf{WfpProj}(R))$ and
$V_X,V^X\in\operatorname{CE}(\mathsf{SfpInj}(R))$. For every $n$, applying any
of the functors $(-)_n$, $Z_n$, $B_n$ and $H_n$ gives, respectively,
a special $\mathsf{WfpProj}(R)$-precover sequence and a special
$\mathsf{SfpInj}(R)$-preenvelope sequence in $R$-Mod.
\end{Corollary}
\begin{proof}
The sequences are provided by Theorem~\ref{ce-complete}. CE-exactness
ensures exactness after each of the four functors. The definition of
the CE classes places the resulting terms in $\mathsf{WfpProj}(R)$ and
$\mathsf{SfpInj}(R)$. For example,
\[
0\to H_n(V_X)\to H_n(U_X)\to H_n(X)\to0
\]
has middle term in $\mathsf{WfpProj}(R)$ and kernel in $\mathsf{SfpInj}(R)$.
Applying $\operatorname{Hom}_R(A,-)$, for $A\in\mathsf{WfpProj}(R)$, and using
$\operatorname{Ext}^1_R(A,H_n(V_X))=0$ proves the precover property.
For the second sequence, apply $\operatorname{Hom}_R(-,B)$ for
$B\in\mathsf{SfpInj}(R)$ and use
$\operatorname{Ext}^1_R(H_n(U^X),B)=0$. The other three functors give
the same argument. Their functoriality also preserves the functorial
choice of these approximation sequences.
\end{proof}

The ordinary and CE constructions use the same set $\mathcal U'$,
but give different right orthogonals:
\[
\mathcal U'^{\perp_{1,\operatorname{Ch}(R)}}
=\widetilde{\mathsf{SfpInj}}(R),
\qquad
\mathcal U'^{\perp_{1,\mathrm{CE}}}
=\operatorname{CE}(\mathsf{SfpInj}(R)).
\]
In the ordinary structure, the tests $S^n(R)$ force acyclicity. In the
CE structure these spheres are projective, and the remaining tests
impose the component and cycle conditions of
Proposition~\ref{prop1}.

This difference is reflected in the approximations. The ordinary
strongly FP-injective preenvelope of
Corollary~\ref{ordinary-approximations} has acyclic target, so its
sequence cannot be CE-exact when the source has nonzero homology.
The CE preenvelope instead induces a special preenvelope on each
homology module. The comparison simplifies in the case of acyclic complexes. To see that, we start by stating a lemma that follows from \cite[Proposition 10.1]{En} and \cite[Lemma 2.10]{LuLiuCE}.

\begin{Lemma}\label{ce-ordinary-comparison}
If either $X$ or $Y$ is acyclic, the natural map
\[
\operatorname{Ext}^1_{\mathrm{CE}}(X,Y)
\longrightarrow\operatorname{Ext}^1_{\operatorname{Ch}(R)}(X,Y)
\]
is an isomorphism.
\end{Lemma}

\begin{Proposition}\label{ce-dg-comparison}
There are inclusions
\[
\begin{gathered}
\widetilde{\mathsf{WfpProj}}(R)\subseteq\operatorname{CE}(\mathsf{WfpProj}(R))
\subseteq\operatorname{dg}\mathsf{WfpProj}(R),\\
\widetilde{\mathsf{SfpInj}}(R)\subseteq\operatorname{CE}(\mathsf{SfpInj}(R))
\subseteq\operatorname{dg}\mathsf{SfpInj}(R).
\end{gathered}
\]
Moreover,
\[
\operatorname{CE}(\mathsf{WfpProj}(R))\cap\mathsf{Ex}=\widetilde{\mathsf{WfpProj}}(R),
\qquad
\operatorname{CE}(\mathsf{SfpInj}(R))\cap\mathsf{Ex}=\widetilde{\mathsf{SfpInj}}(R).
\]
\end{Proposition}
\begin{proof}
The first inclusions and the equalities follow from the definitions,
extension closure of both module classes, and $B_nX=Z_nX$, $H_nX=0$
for an acyclic $X$. If $P\in\operatorname{CE}(\mathsf{WfpProj}(R))$ and
$B\in\widetilde{\mathsf{SfpInj}}(R)$, Theorem~\ref{ce-complete} and
Lemma~\ref{ce-ordinary-comparison} give
$\operatorname{Ext}^1_{\operatorname{Ch}(R)}(P,B)=0$. Hence
$P\in{}^{\perp_1}\widetilde{\mathsf{SfpInj}}(R)=\operatorname{dg}\mathsf{WfpProj}(R)$.
Dually, for $A\in\widetilde{\mathsf{WfpProj}}(R)$ and
$I\in\operatorname{CE}(\mathsf{SfpInj}(R))$ the same lemma gives ordinary
first-Ext vanishing, so
$I\in\widetilde{\mathsf{WfpProj}}(R)^{\perp_1}=\operatorname{dg}\mathsf{SfpInj}(R)$.
\end{proof}

\begin{Corollary}\label{acyclic-pair}
The pair $(\widetilde{\mathsf{WfpProj}}(R),\widetilde{\mathsf{SfpInj}}(R))$ is a
complete hereditary cotorsion pair in the exact category $\mathsf{Ex}$
with the short exact sequences inherited from $\operatorname{Ch}(R)$.
This inherited exact structure is also the restriction of
$\mathcal E_{\mathrm{CE}}$.
\end{Corollary}
\begin{proof}
The category $\mathsf{Ex}$ is extension closed, so its first Ext groups
are the ambient first Ext groups between acyclic objects. Theorem
\ref{induced-pairs} and its compatibility equalities identify the two
orthogonals within $\mathsf{Ex}$ with $\widetilde{\mathsf{WfpProj}}(R)$ and
$\widetilde{\mathsf{SfpInj}}(R)$.

For $X\in\mathsf{Ex}$, the first and third sequences of
Corollary~\ref{ordinary-approximations} have all three terms acyclic.
Compatibility gives $P_X\in\widetilde{\mathsf{WfpProj}}(R)$ and
$I_X\in\widetilde{\mathsf{SfpInj}}(R)$, proving both halves of completeness.
An ambient projective resolution of an acyclic complex has acyclic
syzygies, because projective complexes are acyclic. Consequently, it is 
also a projective resolution in $\mathsf{Ex}$. Therefore, higher Ext groups
in $\mathsf{Ex}$ agree with ambient Ext groups, and hereditarity follows
from Theorem~\ref{induced-pairs}. Finally,
Lemma~\ref{ce-ordinary-comparison} shows that every short exact sequence
of acyclic complexes is CE-exact.
\end{proof}

\begin{Corollary}\label{ce-dg-regular}
The equality $\operatorname{CE}(\mathsf{SfpInj}(R))=\operatorname{dg}\mathsf{SfpInj}(R)$
holds if and only if $R$ is von Neumann regular.
\end{Corollary}
\begin{proof}
Suppose the equality holds. For any module $M$, apply the third
sequence of Corollary~\ref{ordinary-approximations} to $S^0(M)$.
It gives a quasi-isomorphism $S^0(M)\to I$ with
$I\in\operatorname{dg}\mathsf{SfpInj}(R)=\operatorname{CE}(\mathsf{SfpInj}(R))$.
Hence $M\cong H_0(I)\in\mathsf{SfpInj}(R)$. Thus every module belongs to
$\mathsf{SfpInj}(R)$, and every finitely presented module is projective.
This is the standard characterization of a von Neumann regular ring
\cite{Lam}.

Conversely, over a von Neumann regular ring every finitely presented
module is projective, so $\mathsf{SfpInj}(R)=R$-Mod and $\mathsf{WfpProj}(R)=Proj$.
Every object of $\widetilde{\mathsf{WfpProj}}(R)$ is contractible, since its
cycle sequences split. Its Hom-complex into any complex is therefore
contractible. Hence $\operatorname{dg}\mathsf{SfpInj}(R)=\operatorname{Ch}(R)
=\operatorname{CE}(\mathsf{SfpInj}(R))$.
\end{proof}

\subsection{The dg-exact structure}

The preceding constructions lead to two ordinary cotorsion pairs,
one with dg-weakly FP-projective left class and the other with
dg-strongly FP-injective right class. We now place these two dg
classes together in a cotorsion pair by changing the exact structure.
The relevant structure comes from the balanced pair of dg-projective
and dg-injective complexes; compare \cite[Remark 10.4]{En} and
\cite[Section 4]{EPZBalanced}. Here $\operatorname{dg}\mathrm{Proj}$ consists of complexes of projective modules
whose Hom-complex into every acyclic complex is acyclic, and
$\operatorname{dg}\mathrm{Inj}$ is defined dually. The ordinary cotorsion pairs
$(\operatorname{dg}\mathrm{Proj},\mathsf{Ex})$ and
$(\mathsf{Ex},\operatorname{dg}\mathrm{Inj})$ are complete hereditary and form a
cotorsion triplet. Consequently, the short exact sequences preserved
by $\operatorname{Hom}_{\operatorname{Ch}(R)}(P,-)$ for every
$P\in \operatorname{dg}\mathrm{Proj}$ are precisely those preserved by
$\operatorname{Hom}_{\operatorname{Ch}(R)}(-,I)$ for every
$I\in \operatorname{dg}\mathrm{Inj}$. They define an exact structure
$\mathcal E_{\mathrm{dg}}$ with enough projectives and injectives,
whose projective and injective objects are $\operatorname{dg}\mathrm{Proj}$ and
$\operatorname{dg}\mathrm{Inj}$, respectively; see also
\cite[Theorem 2.9]{WLHExact}. We write
$\operatorname{Ext}^i_{\mathcal E_{\mathrm{dg}}}$ for its Yoneda Ext. We first describe these conflations in the ordinary derived category.

\begin{Proposition}\label{dg-derived-split}
Let $0\to Y\to E\xrightarrow{p}X\to0$ be a short exact sequence
of complexes, and let $\delta:X\to\Sigma Y$ be the connecting
morphism of its associated triangle in $D(R)$. The sequence is an
$\mathcal E_{\mathrm{dg}}$-conflation if and only if $\delta=0$;
equivalently, its triangle splits in $D(R)$. In particular,
\begin{enumerate}[label=(\roman*)]
\item $\mathcal E_{\mathrm{dg}}\subseteq\mathcal E_{\mathrm{CE}}$,
and each induced short exact sequence on homology is split;
\item if either $X$ or $Y$ is acyclic, the natural map
\[
\operatorname{Ext}^1_{\mathcal E_{\mathrm{dg}}}(X,Y)
\longrightarrow\operatorname{Ext}^1_{\operatorname{Ch}(R)}(X,Y)
\]
is an isomorphism;
\item for modules $M,N$ and $n\in\mathbb Z$,
$\operatorname{Ext}^1_{\mathcal E_{\mathrm{dg}}}(S^n(M),S^n(N))=0$.
\end{enumerate}
\end{Proposition}
\begin{proof}
Choose a short exact sequence $0\to K\to P\xrightarrow{q}X\to0$
with $P\in \operatorname{dg}\mathrm{Proj}$ and $K$ acyclic. If the given sequence
is an $\mathcal E_{\mathrm{dg}}$-conflation, $q$ lifts through $p$.
Thus $\delta q=0$ in $D(R)$. Since $q$ is a quasi-isomorphism,
$\delta=0$.

Conversely, for every $P\in \operatorname{dg}\mathrm{Proj}$ there are natural
isomorphisms
\[
\operatorname{Ext}^1_{\operatorname{Ch}(R)}(P,Y)
\cong\operatorname{Hom}_{K(R)}(P,\Sigma Y)
\cong\operatorname{Hom}_{D(R)}(P,\Sigma Y).
\]
The first holds because every extension with quotient $P$ is split
in each degree, and degreewise split extensions are classified by
maps to the suspension up to homotopy. The second follows from the
dg-projectivity of $P$. Under these identifications, the pullback
of the given extension along a chain map $f:P\to X$ corresponds
to $\delta f$. If $\delta=0$, every such pullback splits, so every
$f$ lifts through $p$. This is the defining condition for a
conflation.

A split triangle gives split short exact sequences on homology.
For a short exact sequence of complexes, vanishing of the homology
connecting maps also gives surjectivity on cycles: if $x\in Z_nX$
is lifted to $e\in E_n$, then $d(e)\in Z_{n-1}Y$, and the vanishing
connecting class allows us to subtract an element of $Y_n$ to make
$e$ a cycle. This proves (i). If either end is acyclic, it is zero
in $D(R)$ and hence $\delta=0$. All ordinary extensions are then
conflations, with the same equivalence relation and Baer sum, proving
(ii). Finally, an extension of stalks in degree $n$ is itself a stalk
extension. By (i), its sequence of modules splits, proving (iii).
\end{proof}
The derived-category description also explains a limitation of this
exact structure. Its deflations are split epimorphisms on homology,
so an exact-category generator would impose a strong generation
condition on the module category. In general, the efficiency argument
used for CE approximations is therefore unavailable.
\begin{Proposition}\label{dg-generator-obstruction}
If $(\operatorname{Ch}(R),\mathcal E_{\mathrm{dg}})$ has a generator
in the exact-category sense, then $R$ is left pure semisimple.
Consequently, if $R$ is not left pure semisimple, this exact category
is not efficient.
\end{Proposition}
\begin{proof}
Let $G$ be such a generator. For every module $M$ there is an
$\mathcal E_{\mathrm{dg}}$-deflation
$G^{(J)}\twoheadrightarrow S^0(M)$ for some set $J$.
By Proposition~\ref{dg-derived-split}, its map on zeroth homology
is a split epimorphism. Since homology commutes with coproducts,
$M$ is a direct summand of $H_0(G)^{(J)}$. Therefore
\[
R\text{-}\mathrm{Mod}=\operatorname{Add}(H_0(G)),
\]
where $\operatorname{Add}$ denotes direct summands of arbitrary
coproducts. By \cite[Proposition 2.2]{StovicekLocally}, applied to
left modules, this implies that $R$ is left pure semisimple.
Efficiency requires an exact-category generator by axiom (Ef3)
of \cite[Definition 3.4]{StovicekExact}.
\end{proof}
The absence of an exact-category generator does not prevent the
cotorsion pair of interest from being complete. Instead of using
set-cogeneration in the dg exact structure, we use the ordinary
induced pairs and identify their relative orthogonals. The argument
applies to any complete hereditary module cotorsion pair.
\begin{Theorem}\label{dg-relative-pair}
Let $(\mathcal L,\mathcal R)$ be a complete hereditary cotorsion
pair in $R$-Mod, and use the induced classes
$\widetilde{\mathcal L},\widetilde{\mathcal R}$,
$\operatorname{dg}\mathcal L,\operatorname{dg}\mathcal R$
defined as in Section 3. Then
$(\operatorname{dg}\mathcal L,\operatorname{dg}\mathcal R)$ is a
complete hereditary cotorsion pair in
$(\operatorname{Ch}(R),\mathcal E_{\mathrm{dg}})$.
In particular, for every ring $R$ the pair
\[
\bigl(\operatorname{dg}\mathsf{WfpProj}(R),
      \operatorname{dg}\mathsf{SfpInj}(R)\bigr)
\]
is complete hereditary in this exact category.
\end{Theorem}
\begin{proof}
The ordinary pairs
$(\operatorname{dg}\mathcal L,\widetilde{\mathcal R})$ and
$(\widetilde{\mathcal L},\operatorname{dg}\mathcal R)$ are complete
hereditary, and satisfy
$\operatorname{dg}\mathcal L\cap\mathsf{Ex}
=\widetilde{\mathcal L}$ and its dual, by
\cite[Corollaries 3.8 and 3.13]{GillespieFlat} and
\cite[Theorem 3.5]{YangLiu}.

We first prove relative orthogonality. Let
$X\in\operatorname{dg}\mathcal L$,
$Y\in\operatorname{dg}\mathcal R$, and consider a conflation
$0\to Y\xrightarrow{i}E\xrightarrow{p}X\to0$.
Choose $0\to K\xrightarrow{j}P\xrightarrow{q}X\to0$ with
$P\in \operatorname{dg}\mathrm{Proj}$ and $K$ acyclic.
We have $\operatorname{dg}\mathrm{Proj}\subseteq\operatorname{dg}\mathcal L$:
projective modules belong to $\mathcal L$, and
$\widetilde{\mathcal R}\subseteq\mathsf{Ex}$.
Since $\operatorname{dg}\mathcal L$ is resolving in the ordinary
abelian structure, $K\in\operatorname{dg}\mathcal L\cap\mathsf{Ex}
=\widetilde{\mathcal L}$.

The defining lifting property of the conflation provides a chain
map $s:P\to E$ with $ps=q$. Its restriction to $K$ factors as
$sj=if$ for a chain map $f:K\to Y$.
The definition of $\operatorname{dg}\mathcal R$ implies that $f$
is null-homotopic. Write
\[
f_n=d^Y_{n+1}h_n+h_{n-1}d^K_n,
\qquad h_n:K_n\longrightarrow Y_{n+1}.
\]
Since $X_n\in\mathcal L$ and $Y_{n+1}\in\mathcal R$, the equality
$\operatorname{Ext}^1_R(X_n,Y_{n+1})=0$ allows each $h_n$ to extend
to a map $\widehat h_n:P_n\to Y_{n+1}$. The maps
\[
g_n=d^Y_{n+1}\widehat h_n+\widehat h_{n-1}d^P_n
\]
form a chain map $g:P\to Y$ with $gj=f$. Thus $s-ig$ vanishes
on $K$ and factors through a chain map $t:X\to E$.
The equality $ptq=q$ and the epimorphism $q$ give $pt=1_X$.
Hence the conflation splits.

If $V$ is right orthogonal to $\operatorname{dg}\mathcal L$
in $\mathcal E_{\mathrm{dg}}$, test against
$\widetilde{\mathcal L}\subseteq\operatorname{dg}\mathcal L$.
Proposition~\ref{dg-derived-split}(ii) gives
$\operatorname{Ext}^1_{\operatorname{Ch}(R)}
(\widetilde{\mathcal L},V)=0$, so
$V\in\operatorname{dg}\mathcal R$. Together with the preceding
vanishing this proves
\[
(\operatorname{dg}\mathcal L)^{\perp_{1,\mathcal E_{\mathrm{dg}}}}
=\operatorname{dg}\mathcal R.
\]
Dually, testing against
$\widetilde{\mathcal R}\subseteq\operatorname{dg}\mathcal R$
identifies the left orthogonal with $\operatorname{dg}\mathcal L$.

Now apply \cite[Corollary 3.6(1)]{WLHExact} to the ordinary complete
hereditary pair
$(\operatorname{dg}\mathcal L,\widetilde{\mathcal R})$.
Its left class contains the $\mathcal E_{\mathrm{dg}}$-projectives,
as checked above. The cited result gives completeness and heredity
for the relative pair with this left class, whose right class we
have just identified.

More explicitly, the ordinary pairs supply sequences
\[
0\to V_X\to U_X\to X\to0,
\qquad
0\to X\to V^X\to U^X\to0,
\]
with $U_X\in\operatorname{dg}\mathcal L$,
$V_X\in\widetilde{\mathcal R}$,
$V^X\in\operatorname{dg}\mathcal R$ and
$U^X\in\widetilde{\mathcal L}$. Both are conflations by
Proposition~\ref{dg-derived-split}(ii), because their kernel and
cokernel, respectively, are acyclic. These give the required special
approximations directly. The particular assertion follows from
Proposition~\ref{module-pair}.
\end{proof}

\section{The accessible category of Strongly FP-injective complexes}

We remind the reader that a complex is called finitely presented if it is bounded and consists of finitely presented modules and that an object $F$ in a category $\mathcal{K}$ is called finitely presentable if the functor $\operatorname{Hom}_\mathcal{K}(F,-):\mathcal{K}\longrightarrow \textbf{Set}$ preserves directed colimits {\rm(\cite[Definition 1.1]{AR})}. A category $\mathcal{K}$ is called locally finitely presentable provided that it is cocomplete (i.e. it has all small colimits) and there exists a set $\mathcal{F}\subset\mathcal{K
}$ of finitely presentable objects such that every object in $\mathcal{K}$ is a directed colimit of objects from $\mathcal{F}$ {\rm(\cite[Definition 1.9]{AR})}. By \cite[Corollary 4.6]{CH}, the category of complexes over a ring is locally finitely presentable and the finitely presentable objects are exactly the finitely presented complexes.

Let $\kappa$ be a regular infinite cardinal, a poset $\Lambda$ is called $\kappa$-directed if every subset of $\Lambda$ of less than $\kappa$ elements has an upper bound in $\Lambda$. A diagram in a category $\mathcal{K}$ is a functor $D:\mathrm{D}\rightarrow\mathcal{K}$ from a small category $\mathrm{D}$ which is called the scheme of the diagram. A diagram is called $\kappa$-directed if its scheme is a $\kappa$-directed poset. Finally, we call an object $K\in\mathcal{K}$ $\kappa$-presentable if the functor $\operatorname{Hom}_\mathcal{K}(K,-):\mathcal{K}\longrightarrow \textbf{Set}$ preserves $\kappa$-directed colimits  {\rm(\cite[Definition 1.13]{AR})}. For example, in $R\text{-Mod}$ the $\kappa$-presentable objects are exactly the $<\kappa$-presented modules. In addition, with this terminology, the $\aleph_0$-presentable objects are the finitely presentable objects.

The category $\mathcal{K}$ is called $\kappa$-accessible if it has all $\kappa$-directed colimits and a set $\mathcal{A}$ of $\kappa$-presentable objects such that every object of $\mathcal{K}$ is a $\kappa$-directed colimit of objects from $\mathcal{A}$. If in addition, the category $\mathcal{K}$ is cocomplete then it is called locally $\kappa$-presentable {\rm(\cite[Definition 2.1 and Definition 1.17]{AR})}. For example, the category of complexes of modules over a ring is locally $\kappa$-presentable for every regular infinite cardinal $\kappa$. In the case that, $\kappa$ is uncountable, the $\kappa$-presentable objects are exactly all (bounded and unbounded) complexes consisting of $\kappa$-presentable objects {\rm(\cite[Lemma 1.5(c)]{PO})}.

\begin{Definition}Let $\lambda$ be an infinite cardinal. We call a ring left $<\lambda$-coherent if, for every $\lambda$-presentable left $R$-module $K$, every submodule of $K$ having fewer than $\lambda$ generators is $\lambda$-presentable.  
    
\end{Definition}

\begin{Remark} For $\lambda=\aleph_0$ the above definition retrieves the definition of coherence.
    
\end{Remark}

\vspace{0.1in}

For two cardinals $\nu$ and $\lambda$ we let $\nu^{<\lambda}$ denote the supremum of the cardinals $\nu^\mu$ taken over all the cardinals $\mu<\lambda$.

\vspace{0.1in}

\begin{Proposition}\label{Posfpinj} 
Let $R$ be a left $<\lambda$-coherent ring where $\lambda$ is a regular infinite cardinal. Let $\nu$ be a cardinal such that $|R| + \aleph_0\leq \nu $ and $\nu^{<\lambda}=\nu$. Then the full subcategory of strongly FP-injective modules is closed under $\lambda$-directed colimits in $R$-Mod. The category of strongly FP-injective modules is $\nu^{+}$-accessible with direct limits of $\lambda$-index chains, and the $\nu^{+}$-presentable objects are exactly all the $\nu^{+}$-presentable strongly FP-injective modules. Moreover, every strongly FP-injective module is a $\nu^{+}$-directed colimit of $\nu^{+}$-presentable strongly FP-injective modules. More precisely, for every strongly FP-injective module $M$, the poset, $\mathcal{P}$, of all $\nu^{+}$-presentable strongly FP-injective submodules of $M$ is $\nu^{+}$-directed and $M$ can be viewed as the $\nu^{+}$-directed union of $\mathcal{P}$.   
    
\end{Proposition}

\begin{proof} Let $\mathcal{T}$ be as in Proposition \ref{module-pair}. Since $R$ is $<\lambda$-coherent every finitely presented module $F$ has a projective resolution \[\cdots\rightarrow P_m\rightarrow\cdots\rightarrow P_1\rightarrow P_0\rightarrow F\rightarrow 0\] with all $P_i$ having fewer than $\lambda$ generators. Using these suitable projective resolutions to define the set $\mathcal{T}$ of Proposition \ref{module-pair} the statement follows by setting $\mathbf{S}=\mathcal{T}$ in \cite[Proposition 3.6]{PO}.

\end{proof}

\begin{Lemma} The category of pure acyclic complexes over a ring is $\aleph_0$-accessible. The finitely presentable complexes of this category are the bounded contractible complexes of finitely presented modules. 
    
\end{Lemma}

\begin{proof} It is well known that the pure acyclic complexes are closed under directed colimits and that they contain all contractible complexes. In the introduction of this section we already mentioned that the bounded complexes of finitely presented modules are the finitely presentable objects of the full category of complexes, so, in particular, the bounded contractible complexes of finitely presented modules are finitely presentable objects in the category of pure acyclic complexes. To establish that every pure acyclic complex is a directed colimit of bounded contractible complexes of finitely presented modules we invoke \cite[Proposition 2.2(vii)]{E}. Finally, every finitely presentable object in the category of pure acyclic complexes must be a direct summand of a bounded contractible complex of finitely presented modules (see, for example, \cite[Proposition 1.1]{PO}) and since the bounded contractible complexes of finitely presented modules are closed under direct summands we deduce that the finitely presentable objects in our category of interest are exactly those described in the statement.
\end{proof}

\begin{Definition} For two regular infinite cardinals $\lambda$ and $\mu$ with $\lambda<\mu$ we say that $\lambda$ is sharply smaller than $\mu$ if every $\lambda$-accessible category is $\mu$-accessible. We write $ \lambda\triangleleft\mu$.
    
\end{Definition}
    
\begin{Corollary} For every uncountable infinite regular cardinal $\kappa$ the category of pure acyclic complexes is $\kappa$-accessible. The $\kappa$-presentable objects are exactly the (unbounded and bounded) pure acyclic complexes of $<\kappa$-presented modules.
    
\end{Corollary}

    \begin{proof} Since the category of pure acyclic complexes is $\aleph_0$-accessible and $\aleph_0$ is sharply smaller than every regular uncountable cardinal $\kappa$ by \cite[Example 2.13(1)]{AR}, we have that the category of pure acyclic complexes is $\kappa$-accessible for every uncountable regular cardinal $\kappa$. 

    \cite[Lemma 1.5(c)]{PO} informs us that the $\kappa$-presentable objects in the category of complexes are exactly the (unbounded and bounded) complexes of $<\kappa$-presented modules. This means, in particular, that all (unbounded and bounded) pure acyclic complexes of $<\kappa$-presented modules are $\kappa$-presentable in the category of complexes which in turn means that they are $\kappa$-presentable in the category of pure acyclic complexes.
    
    Conversely, since the category of pure acyclic complexes is $\aleph_0$-accessible we know that every $\kappa$-presentable object in this category is a split subobject of a $\kappa$-small directed colimit of  finitely presentable objects (see \cite[Remark 2.15]{AR}). Since the finitely presentable objects in the category of pure acyclic complexes are also finitely presentable in the larger category of complexes, by \cite[Proposition 1.16]{AR}, we get that every $\kappa$-small directed colimit of finitely presentable objects (finitely presentable objects are $\lambda$-presentable for each regular infinite cardinal $\lambda$) is $\kappa$-presentable. So, every $\kappa$-presentable object in the category of pure acyclic complexes is a split subobject of a $\kappa$-presentable object in the category of complexes. Finally, it is easy to see that every split subobject of a $\kappa$-presentable object in the category of complexes is in fact $\kappa$-presentable.
        
    \end{proof}

\begin{Proposition}[Pseudopullback]\label{Pseudopullback}
Let $\nu$ be a regular cardinal and $\lambda$ be a smaller infinite cardinal. Consider the $\nu$-accessible categories $\mathcal{A},\mathcal{B}$ and $C$ in which all $\lambda$-indexed chains (of objects and morphisms) have colimits. Let $\Theta_\mathcal{A}:\mathcal{A}\rightarrow C$  and $\Theta_\mathcal{B}:\mathcal{B}\rightarrow C$ be two functors preserving $\nu$-directed colimits and colimits of $\lambda$-indexed chains, and taking $\nu$-presentable objects to $\nu$-presentable objects. Then the pseudopullback category, $\mathbf{P}$, of $\Theta_\mathcal{A}$ and $\Theta_\mathcal{B}$ is defined as the category of triples $(A,B,\phi)$, where $A\in\mathcal{A},B\in\mathcal{B}$ and $\phi:\Theta_\mathcal{A}(A)\rightarrow\Theta_\mathcal{B}(B)$ is an isomorphism in $C$. The pseudopullback  $\mathbf{P}$ is a $\nu$-accessible category. The $\nu$-presentable objects are exactly all the triples $(A,B,\phi)$ such that $A$ is a $\nu$-presentable object in $\mathcal{A}$ and $B$ is a $\nu$-presentable object in $\mathcal{B}$.

\end{Proposition}

\begin{proof} We refer the reader to \cite[Corollary 5.1]{POlimits}. The following references are also appropriate: \cite[Remark 3.2(I), Theorem 3.8, Corollary 3.9 and Remark 3.11(II)]{Ulmer}, \cite[Theorem 2.2]{Raptis} with the proof in \cite[Proposition 3.1]{Chorny}.  
    
\end{proof}

\begin{Theorem}\label{pacofsfpinj} Let $R$ be a left $<\lambda$-coherent ring where $\lambda$ is a regular cardinal. Let $\nu$ be a cardinal such that $|R| + \aleph_0\leq \nu $ and $\nu^{<\lambda}=\nu$. Then the category of strongly FP-injective complexes is closed under $\lambda$-directed colimits. In addition, the category of strongly FP-injective complexes is $\nu^+$-accessible. The $\nu^+$-presentable objects are exactly the strongly FP-injective complexes of $<\nu^+$-presented modules.
    
\end{Theorem}

\begin{proof} Let $C=\prod_{n\in\mathbb{Z}}R\text{-}Mod$. Also let $\Theta_{1}:Pureac\rightarrow C$ be the functor from the category of pure acyclic complexes to the category $C$ that sends a pure acyclic complex $A$ to the collection $(A_n)_{n\in\mathbb{Z}}$. Now consider the inclusion functor $\Theta_{2}:\prod_{n\in\mathbb{Z}}\mathsf{SfpInj}(R)\rightarrow C$ from the category of $\mathbb{Z}$ copies of the class of strongly FP-injective modules to $C$. Then the pseudopullback of $\Theta_1$ and $\Theta_2$ is equivalent to the category of strongly FP-injective complexes.

The category $C$ is $\nu^+$-accessible by \cite[Proposition 1.2]{PO}, the category $\prod_{n\in\mathbb{Z}}\mathsf{SfpInj}(R)$ is $\nu^+$-accessible by \cite[Proposition 1.2]{PO} and Proposition \ref{Posfpinj}. In addition, the pseudopullback of $\Theta_1$ and $\Theta_2$ being equivalent to the category of strongly FP-injective complexes stems from the description of strongly FP-injective complexes as pure acyclic complexes of strongly FP-injective modules given in Theorem \ref{Theor}.

Lastly, since pure acyclic complexes are closed under directed colimits and strongly FP-injective modules are closed under $\lambda$-directed colimits by Proposition \ref{Posfpinj} strongly FP-injective complexes must be closed under $\lambda$-directed colimits.
    
\end{proof}

\begin{Remark}
    (i) The above Theorem can be viewed as an analogue of the toy example \cite[Corollary 10.14.]{POlimits} that states that: For any ring $R$ and any uncountable regular cardinal $\kappa$ any pure acyclic complex of flat modules is a $\kappa$-directed colimit of pure acyclic complexes of $\kappa$-presentable flat modules.
    
    (ii) For this analogy to become a duality via Theorem \ref{pacofsfpinj} we need to assume that the ring $R$ is left coherent and of countable cardinality. 
\end{Remark}

\begin{Corollary}\label{Dualwithflats} Let $R$ be a countable left coherent ring and $\kappa$ be any uncountable regular cardinal. Then every pure acyclic complex of absolutely pure modules (In this case absolutely pure modules coincide with strongly FP-injective modules.) is a $\kappa$-directed colimit of pure acyclic complexes of $\kappa$-presentable absolutely pure modules.
    
\end{Corollary} 

\begin{proof} Use $\lambda=\nu=\aleph_0$ in Theorem \ref{pacofsfpinj} to ensure that the category of pure acyclic complexes of absolutely pure modules is $\aleph_1$-accessible. Since, for coherent rings, the category of pure acyclic complexes of absolutely pure modules is closed under directed colimits the statement for any uncountable regular cardinal $\kappa$ stems from \cite[proof of Proposition 4.1.]{Beke}.  
    
\end{proof}

It is worth noting that imposing on the ring $R$ the above assumptions also implies that all absolutely pure $R$-modules have injective dimension at most one. Indeed, since $R$ is left coherent all absolutely pure modules are strongly FP-injective modules, while $R$ being countable means that every strongly FP-injective module has injective dimension at most one (see Corollary \ref{Injdimofsfpinj}). In addition, in this case flat modules have projective dimension at most 1 (see Remark \ref{simson}(ii)).

\section{The accessible category of CE-Strongly FP-injective complexes}

In this section we turn our attention to the accessible category of CE-strongly FP-injective complexes. Before doing so we need to deal with the category of complexes of strongly FP-injective modules.

\begin{Theorem}\label{C(sfpinj)}
Let $R$ be a left $<\lambda$-coherent ring, where $\lambda$
is a regular infinite cardinal. Let $\nu$ be a cardinal
such that
\[
    |R|+\aleph_0\leq \nu
    \qquad\text{and}\qquad
    \nu^{<\lambda}=\nu.
\]
Then the full subcategory
$\operatorname{Ch}(\mathsf{SfpInj}(R))$
of complexes of strongly FP-injective modules is closed
under $\lambda$-directed colimits in $\operatorname{Ch}(R)$
and is $\nu^+$-accessible.

Its $\nu^+$-presentable objects are precisely the complexes
whose components are $\nu^+$-presentable strongly
FP-injective modules. Consequently, every complex of
strongly FP-injective modules is a $\nu^+$-directed colimit
of complexes of $<\nu^+$-presented strongly FP-injective
modules.
\end{Theorem}

\begin{proof}
Put $\kappa=\nu^+$. The cardinal assumptions imply
$\lambda\leq\nu<\kappa$, and $\kappa$ is uncountable.
Consider the functors
\[
    U:\operatorname{Ch}(R)
      \longrightarrow \prod_{n\in\mathbb Z}R\text{-Mod},
    \qquad
    U(X)=(X_n)_{n\in\mathbb Z},
\]
and
\[
    V:\prod_{n\in\mathbb Z}\mathsf{SfpInj}(R)
      \longrightarrow \prod_{n\in\mathbb Z}R\text{-Mod},
\]
where $V$ is the inclusion. Their pseudopullback is
equivalent to $\operatorname{Ch}(\mathsf{SfpInj}(R))$.

The category $\operatorname{Ch}(R)$ is locally
$\kappa$-presentable. By Proposition \ref{Posfpinj} and the
accessibility theorem for products, the other two
categories are $\kappa$-accessible. Here we use
$|\mathbb Z|<\kappa$.

All three categories admit $\kappa$-directed colimits
and colimits of $\lambda$-indexed chains, and both
functors preserve these colimits, since they are
computed componentwise. Moreover, both functors
preserve $\kappa$-presentable objects: for $U$, this
follows from the componentwise characterization of
$\kappa$-presentable complexes, and for $V$, from
Proposition \ref{Posfpinj} and the corresponding characterization
for products.

The pseudopullback Theorem therefore gives the asserted
accessibility and description of presentable objects.
Finally, closure under $\lambda$-directed colimits
follows componentwise from Proposition \ref{Posfpinj}.
\end{proof}

\begin{Remark}
    (i)\em{(\cite[Corollary 10.4]{POlimits})} There exists an analogous result dealing with flat modules, namely, for any ring $R$ and any uncountable regular cardinal $\kappa$ any complex of flat $R$-modules is a $\kappa$-directed colimit of complexes of $\kappa$-presentable flat modules.
    
    (ii) For this analogy to become a duality via the above Theorem we need to assume that the ring $R$ is left coherent and of countable cardinality. 
\end{Remark}

\begin{Corollary} Let $R$ be a countable left coherent ring and $\kappa$ be any uncountable regular cardinal. Then every complex of absolutely pure modules (In this case absolutely pure modules coincide with strongly FP-injective modules.) is a $\kappa$-directed colimit of complexes of $\kappa$-presentable absolutely pure modules.
\end{Corollary}

\begin{proof}
    The reader can follow the reasoning of Corollary \ref{Dualwithflats}.
\end{proof}

We note again that with the above assumptions all absolutely pure modules have injective dimension at most one.

\begin{Theorem}\label{CE(sfpinj)}
Let $R$ be a left $<\lambda$-coherent ring, where $\lambda$ is
a regular infinite cardinal. Let $\nu$ be a cardinal such that
\[
    |R|+\aleph_0\leq \nu
    \qquad\text{and}\qquad
    \nu^{<\lambda}=\nu.
\]
Then the full subcategory of CE-strongly FP-injective complexes
is closed under $\lambda$-directed colimits in $\mathrm{Ch}(R)$
and is $\nu^+$-accessible.

Its $\nu^+$-presentable objects are precisely the CE-strongly
FP-injective complexes whose components are
$<\nu^+$-presented modules. Consequently, every CE-strongly
FP-injective complex is a $\nu^+$-directed colimit of
CE-strongly FP-injective complexes of $<\nu^+$-presented modules.
\end{Theorem}

\begin{proof}
Put $\kappa=\nu^+$. The cardinal assumptions imply
$\lambda\leq\nu<\kappa$, and $\kappa$ is uncountable.
Consider the categories
\[
    \mathcal{A}
      =\mathrm{Ch}\bigl(\mathsf{SfpInj}(R)\bigr),
    \qquad
    \mathcal{D}
      =\prod_{n\in\mathbb{Z}}\mathsf{SfpInj}(R),
    \qquad
    \mathcal{C}
      =\prod_{n\in\mathbb{Z}}R\text{-}\mathrm{Mod},
\]
and the functors
\[
    K:\mathcal{A}\longrightarrow\mathcal{C},
    \qquad
    K(X)=\bigl(Z_nX\bigr)_{n\in\mathbb{Z}}
        =\bigl(\ker d_n^X\bigr)_{n\in\mathbb{Z}},
\]
and
\[
    V:\mathcal{D}\longrightarrow\mathcal{C},
\]
where $V$ is the inclusion.

The pseudopullback of $K$ and $V$ is equivalent to the full
subcategory of complexes $X$ satisfying
\[
    X_n,\ Z_nX\in\mathsf{SfpInj}(R)
    \qquad\text{for every }n\in\mathbb{Z}.
\]
Indeed, an object of the pseudopullback consists of a complex
$X\in\mathcal{A}$, a family $(M_n)_{n\in\mathbb{Z}}$ of strongly
FP-injective modules, and an isomorphism
\[
    (Z_nX)_{n\in\mathbb{Z}}
       \xrightarrow{\ \sim\ }(M_n)_{n\in\mathbb{Z}}.
\]
Since $V$ is fully faithful, forgetting this additional data
gives the asserted equivalence. By Proposition \ref{prop1}, this is
precisely the category of CE-strongly FP-injective complexes.

By Theorem \ref{C(sfpinj)}, the category $\mathcal{A}$ is
$\kappa$-accessible. Proposition \ref{Posfpinj} and the accessibility
theorem for products show that $\mathcal{D}$ and
$\mathcal{C}$ are also $\kappa$-accessible. Here we use
$|\mathbb{Z}|<\kappa$.

All three categories admit $\kappa$-directed colimits and
colimits of $\lambda$-indexed chains, computed componentwise
in $R\text{-}\mathrm{Mod}$. The functor $V$ preserves these
colimits. The functor $K$ also preserves them, since filtered
colimits of modules commute with kernels.

We next verify that both functors preserve
$\kappa$-presentable objects. For $V$, this follows from
Proposition \ref{Posfpinj} and the componentwise characterization of
$\kappa$-presentable objects in the product categories.

Let $X$ be a $\kappa$-presentable object of $\mathcal{A}$.
By Theorem \ref{C(sfpinj)}, each $X_n$ is a $<\kappa$-presented module.
Since $|R|\leq\nu$ and $\kappa=\nu^+$, this implies
$|X_n|\leq\nu$, and hence $|Z_nX|\leq\nu$ for every $n$.
Every $R$-module of cardinality at most $\nu$ is
$<\kappa$-presented: it has a presentation with at most $\nu$
generators, and the kernel of the corresponding free
presentation has cardinality at most $\nu$, so it has at most
$\nu$ generators as well. Thus each $Z_nX$ is
$\kappa$-presentable in $R\text{-}\mathrm{Mod}$, and $K(X)$
is $\kappa$-presentable in $\mathcal{C}$.

The pseudopullback theorem, Proposition \ref{Pseudopullback}, now shows that
the category of CE-strongly FP-injective complexes is
$\kappa$-accessible. Its $\kappa$-presentable objects are
precisely those for which $X$ is $\kappa$-presentable in
$\mathcal{A}$ and $(Z_nX)_{n\in\mathbb{Z}}$ is
$\kappa$-presentable in $\mathcal{D}$.
For a CE-strongly FP-injective complex, the preceding
cardinality argument and Proposition \ref{Posfpinj} show that the
second condition follows from the first. Theorem \ref{C(sfpinj)}
therefore gives the asserted description of presentable
objects.

Finally, let $(X_i)_{i\in I}$ be a $\lambda$-directed diagram
of CE-strongly FP-injective complexes. Filtered colimits
commute with kernels, so
\[
    Z_n\left(\varinjlim_{i\in I}X_i\right)
       \cong \varinjlim_{i\in I}Z_nX_i.
\]
By Proposition \ref{Posfpinj}, both the components and the cycles of
$\varinjlim_{i\in I}X_i$ are strongly FP-injective.
Proposition \ref{prop1} therefore implies that this colimit is
CE-strongly FP-injective. The final assertion follows
from $\kappa$-accessibility.
\end{proof}

\begin{Remark} Contrary to the cases of Theorem \ref{pacofsfpinj} and Theorem \ref{C(sfpinj)} there doesn't seem to exist in the literature an analogous result to that of Theorem \ref{CE(sfpinj)} dealing with $CE(Flat)$ complexes. So we state and prove the following.
    
\end{Remark}

\begin{Proposition}
    
 Let $R$ be a ring and $\kappa$ be an infinite regular cardinal. Then any $CE(Flat)$ complex is a $\kappa$-directed colimit of $CE(Pres_{\kappa}(Flat))$ complexes. Here $Pres_{\kappa}(Flat)$ denotes the category of $\kappa$-presentable flat modules, hence $CE(Pres_{\kappa}(Flat))$ contains all the complexes $F$ which are such that $F$, $Z(F)$, $B(F)$ and $H(F)$ are complexes of $\kappa$-presentable flat modules.
    
\end{Proposition}

\begin{proof} The most challenging case for $\kappa=\aleph_0$ has been settled by Enochs in \cite[Theorem 7.2]{En}. We can deduce the rest by \cite[Theorem 2.11]{AR} and \cite[Example 2.13(1)]{AR}.
    
\end{proof}

\section{Strongly FP-injective modules and $<\lambda$-coherent rings}

In this section we briefly study $<\lambda$-coherent rings and their connection to strongly FP-injective modules. We show that a ring is $<\lambda$-coherent if and only if strongly FP-injective modules are closed under $\lambda$-directed colimits. We also show that the supremum of injective dimensions of strongly FP-injective modules is bounded by $m$ for a $<\aleph_m$-Noetherian ring.

\begin{Definition} Let $\lambda$ be an infinite cardinal. We call a ring left $<\lambda$-Noetherian if every submodule of a $<\lambda$-generated left module is $<\lambda$-generated.

\end{Definition}

\begin{Remark} For $\lambda=\aleph_0$ the above definition retrieves the definitions of Noetherian rings.
    
\end{Remark}

\begin{Lemma}\label{l-coherent}
Let $\lambda$ be an infinite regular cardinal. The following conditions are equivalent:

\noindent(i) $R$ is left $<\lambda$-coherent.

\noindent(ii) Every $<\lambda$-generated submodule $K$ of a free module is $\lambda$-presentable.

\noindent(iii) Every finitely generated left ideal of $R$ is $\lambda$-presentable.

\end{Lemma}

\begin{proof} The implications (i)$\Rightarrow$(ii) and (ii)$\Rightarrow$(iii) are obvious.

In order to prove the implication (iii)$\Rightarrow$(ii), suppose that every finitely generated left ideal is $\lambda$-presentable. Then it is easy to see that every finitely generated submodule of a finitely generated free module is $\lambda$-presentable, consequently every finitely generated submodule of a free module is $\lambda$-presentable. For the case $\lambda=\aleph_0$ (ii) is already proven. Suppose that $\lambda$ is uncountable. Choose a generating set $G$ of $K$ with $|G|<\lambda$. Writing $K$ as a directed union of all its submodules with a finite generating set from $G$ we can deduce by \cite[Proposition 1.16]{AR} that $K$ is $\lambda$-presentable.

For the implication (ii)$\Rightarrow$(i), suppose that every $<\lambda$-generated submodule  of a free module is $\lambda$-presentable. In addition, let $\Lambda$ be a $\lambda$-presentable module and $M$ be a $<\lambda$-generated submodule of $\Lambda$. We wish to prove that $M$ is $\lambda$-presentable.

Consider the commutative diagram:
\[
\begin{tikzcd}[row sep=large, column sep=large]
  & 0 \arrow[d]
  & 0 \arrow[d]
  & 0 \arrow[d]
  & \\
0 \arrow[r]
  & K_1 \arrow[r] \arrow[d]
  & K \arrow[r] \arrow[d]
  & K_0 \arrow[r] \arrow[d]
  & 0
  & \\
0 \arrow[r]
  & F_1 \arrow[r] \arrow[d]
  & F_1 \oplus F_0 \arrow[r] \arrow[d]
  & F_0 \arrow[r] \arrow[d]
  & 0 \\
0 \arrow[r]
  & M \arrow[r] \arrow[d]
  & \Lambda \arrow[r] \arrow[d]
  & C \arrow[r] \arrow[d]
  & 0 \\
  & 0
  & 0
  & 0
  &
\end{tikzcd}
\]
where $F_1$ and $F_0$ are $\lambda$-presentable free modules. Then, of course, $F_1\oplus F_0$ is $\lambda$-presentable. Since $M$ is $<\lambda$-generated and $\Lambda$ is $\lambda$-presentable, $C$ must be $\lambda$-presentable. So, $K$ and $K_0$ are $<\lambda$-generated. By (ii) $K$ and $K_0$ are $\lambda$-presentable meaning that $K_1$ is $<\lambda$-generated, hence $M$ is $\lambda$-presentable.
\end{proof}

\begin{Theorem}\label{sfpinj-directed-characterization} Let $\lambda$ be an infinite regular cardinal. The following are equivalent for a ring $R$.

(i) $R$ is a left $<\lambda$-coherent ring.

(ii) Every $\lambda$-directed colimit of strongly FP-injective modules is a strongly FP-injective module.

(iii) Every $\lambda$-directed colimit of injective modules is an absolutely pure module.

\end{Theorem}

\begin{proof} (i)$\Rightarrow$(ii) See Proposition \ref{Posfpinj}.

(ii)$\Rightarrow$(iii) Obvious.

(iii)$\Rightarrow$(i) We follow the proof of \cite[Theorem 3.2]{Stenström}. Let $L$ be a finitely generated left ideal. We will show that $L$ is $\lambda$-presentable which, by Lemma \ref{l-coherent}, means exactly that $R$ is $<\lambda$-coherent. 

Consider any $\lambda$-directed system $M_i$ and a map $\alpha:L\rightarrow \underrightarrow{lim} M_i$. By \cite[Proposition 3.13 and Corollary 5.9]{StovicekExact}, for each $i$ there exists a short exact sequence
\[
0\longrightarrow M_i\longrightarrow E(M_i)\longrightarrow C(M_i)
\longrightarrow0
\]
functorial in $M_i$, with $E(M_i)$ injective. These produce a compatible diagram of short exact sequences. Taking colimits preserves exactness, and assumption (iii) implies
that $\varinjlim_i E(M_i)$ is absolutely pure. Since the cokernel of the inclusion $0\rightarrow L\rightarrow R$ is a finitely presented module, $\alpha$ extends to a map $\beta:R\rightarrow\underrightarrow{lim} E(M_i)$. It is obvious that $\beta$ factors through some $E(M_j)$. The composite $L\rightarrow R\rightarrow E(M_j)\rightarrow E(M_j)/M_j $ goes to zero in the directed colimit so we can assume it is zero itself by taking $j$ large enough. $L\rightarrow R\rightarrow E(M_j)$ factors through $M_j$ which gives the desired factorization of $\alpha$. So, $L$ is $\lambda$-presentable.

\end{proof}

 The preceding characterization can also be expressed
in terms of strongly FP-injective complexes.
\begin{Corollary}\label{complex-directed-characterization}
Let $\lambda$ be an infinite regular cardinal.
The following conditions are equivalent:
\begin{enumerate}[label=(\roman*)]
\item $R$ is left $<\lambda$-coherent.
\item $\widetilde{\mathsf{SfpInj}}(R)$ is closed under
$\lambda$-directed colimits.
\item $\operatorname{Ch}(\mathsf{SfpInj}(R))$ is closed
under $\lambda$-directed colimits.
\item $\operatorname{CE}(\mathsf{SfpInj}(R))$ is closed
under $\lambda$-directed colimits.
\end{enumerate}
\end{Corollary}

\begin{proof}
Assume (i). By
Theorem~\ref{sfpinj-directed-characterization},
strongly FP-injective modules are closed under
$\lambda$-directed colimits.
Filtered colimits of modules are exact and commute
with cycles, boundaries, and homology.
The descriptions of the three classes therefore give
(ii), (iii), and (iv).

Conversely, let $(M_i)$ be a $\lambda$-directed system
of strongly FP-injective modules. Each disk
$D^0(M_i)$ belongs to all three classes, and
\[
    \varinjlim_i D^0(M_i)
       \cong D^0\left(\varinjlim_i M_i\right).
\]
Any one of (ii), (iii), or (iv) implies that
$\varinjlim_i M_i$ is strongly FP-injective.
Theorem~\ref{sfpinj-directed-characterization}
then yields (i).
\end{proof}

The following Lemma is a special case of 
\cite[Proposition 2.1.]{PO}

\begin{Lemma}\label{alephmpresentable} Every $\aleph_m$-presentable module has weakly FP-projective dimension at most $m$.
    
\end{Lemma}

\begin{proof} For an infinite $m$ the statement is trivially true, so assume that $m$ is a non-negative integer. Let $S$ in \cite[Proposition 2.1.]{PO} be a set of representatives of all finitely presented modules. Let $C$ be an $\aleph_m$-presentable module and $D$ be a strongly FP-injective module. Then $\operatorname{Ext}_R^i(C,D)=0$ for all $i>m$. This means that the $m$-th syzygy of every projective resolution of $C$ is a weakly FP-projective module. Hence, the weakly FP-projective dimension of $C$ does not exceed $m$.
    
\end{proof}

\begin{Corollary}\label{Injdimofsfpinj} Let $R$ be a left $<\aleph_m$-Noetherian ring. Then the injective dimension of any strongly FP-injective module is at most $m$.
    
\end{Corollary}

\begin{proof} Since $R$ is a left $<\aleph_m$-Noetherian ring, every left ideal of $R$ is $<\aleph_m$-generated, which means that every cyclic $R$-module is $\aleph_m$-presentable. The statement follows from the above lemma and the Baer criterion for injectivity. 
    
\end{proof}

\begin{Remark}\label{simson}(i) Every ring of cardinality $\aleph_m$ is a $<\aleph_{m+1}$-Noetherian ring. 

(ii) There exists an analogous result of Corollary \ref{Injdimofsfpinj} dealing with flat modules. Namely, if a ring $R$ has cardinality at most $\aleph_m$ then the projective dimension of any flat module is at most $m+1$\em{ (cf. \cite[Theorem]{Simson})}. However, the cardinality hypothesis cannot be replaced by the assumption that $R$ is
$<\aleph_m$-Noetherian. Indeed, let $d\geq1$ be an integer,
choose a field $k$ of cardinality $\aleph_{d-1}$, and put
\[
    R=k[x_1,\ldots,x_d],
    \qquad Q=k(x_1,\ldots,x_d).
\]
The ring $R$ is commutative Noetherian, and hence
$<\aleph_m$-Noetherian for every non-negative integer $m$.
Its field of fractions $Q$ is a flat $R$-module, being
a localization of $R$, whereas
\cite[Corollary~2.59]{OsofskyDimensions} gives  $\operatorname{pd}_R Q=d$.
Taking $d=m+2$ disproves even the bound
$\operatorname{pd}_R F\leq m+1$ for flat modules over
$<\aleph_m$-Noetherian rings. More generally, since $d$
is arbitrary, there is no finite bound depending only
on $m$ that holds for all flat modules over all such rings.

(iii) The analogous result of Corollary \ref{Injdimofsfpinj} for absolutely pure modules fails. Note that absolutely pure modules of injective dimension at most one are strongly FP-injective modules. So, a countable ring $R$ such that all absolutely pure left $R$-modules have injective dimension at most one would be left coherent. For an example of a countable commutative ring that is not coherent, see \cite[Ex. 17 on page 161]{Lam}.

\end{Remark}

\section{Weakly FP-projective periodicity over $<\aleph_{m}$-coherent rings, $m\in\mathbb{N}$}

As already mentioned in the introduction, it has been an open problem since \cite{EK} whether cycles of acyclic complexes of weakly FP-projective modules are always weakly FP-projective. It is easy to see that Corollary \ref{Injdimofsfpinj} gives a positive answer for $<\aleph_m$-Noetherian rings when $m$ is an integer. In this section, we push this result a little further and establish the result over $<\aleph_m$-coherent rings for every integer $m$.

\begin{Definition}

(i) For a regular cardinal $\lambda$, a submodule $A$ of an $R$-module $B$
is said to be \emph{$\lambda$-pure} in $B$ if, for every
$\lambda$-presentable $R$-module $\Lambda$, every morphism
$f\colon \Lambda \to B/A$ can be lifted to $B$.

(ii)
A short exact sequence
\[
0 \longrightarrow A \longrightarrow B \longrightarrow C
  \longrightarrow 0
\]
is said to be \emph{$\lambda$-pure} if the image of $A$ is
$\lambda$-pure in $B$.

(iii) A complex $C$ is said to be $\lambda$-pure acyclic if it is acyclic and for each $n\in\mathbb{Z}$ the short exact sequence:
\[0\rightarrow Z_n(C)\rightarrow C_n\rightarrow B_{n-1}(C)\rightarrow 0\] is $\lambda$-pure.

\end{Definition}

\begin{Example}\label{exampleofNmpureacyclic}
Let $S$ be a strongly FP-injective module, and let $E$ denote an
injective resolution of $S$:
\[
E = 0 \longrightarrow S \longrightarrow I^0
  \longrightarrow I^1 \longrightarrow I^2
  \longrightarrow \cdots.
\]
Denote by $S^m$ the $m$-th cosyzygy of $S$.
 Then, by Lemma \ref{alephmpresentable}, the complex
\[
E^m =0 \longrightarrow S^m \longrightarrow I^m
  \longrightarrow I^{m+1} \longrightarrow \cdots
\]
is an $\aleph_m$-pure acyclic complex.
\end{Example}

\begin{Definition}
An $R$-module $M$ is \emph{$\lambda$-pure projective} if and only if
it is a direct summand of a direct sum of $\lambda$-presentable modules.
Equivalently, $\operatorname{Hom}_R(M,-)$ is exact with respect to
$\lambda$-pure short exact sequences.
\end{Definition}

The following theorem appears as Theorem 1.3 in \cite{SDW}. 

\begin{Theorem}{\emph{(\cite[Theorem 1.3]{SDW})}}
A complex $E$ in a locally $\lambda$-presentable additive category
$\mathcal{A}$ is $\lambda$-pure acyclic if and only if any chain map
$f\colon X \to E$ from a complex of $\lambda$-pure projective objects
in $\mathcal{A}$ to $E$ is null-homotopic.
\end{Theorem}

As a special case, we obtain the following.

\begin{Theorem}\label{nullhomotopic}
A complex $E$ in $R\text{-}\mathrm{Mod}$ is $\lambda$-pure acyclic
if and only if any chain map $f\colon X \to E$ from a complex of
$\lambda$-pure projective modules to $E$ is null-homotopic.
\end{Theorem}

\begin{Example}
Let $E^m$ be as in Example \ref{exampleofNmpureacyclic}, and let $X$ be a complex of
$\aleph_m$-presentable modules. Then any chain map
$f\colon X \to E^m$ is null-homotopic.
\end{Example}

Before stating the main results of this section we  introduce the following cotorsion pairs.

\begin{Definition}\label{WNmP} Let $\mathcal{C}\aleph_m\mathcal{P}$ be a set of representatives of all $\aleph_m$-presentable weakly FP-projective modules. Consider the cotorsion pair
\[\bigl( ^{\perp_1}((\mathcal{C}\aleph_m\mathcal{P})^{\perp_1}),(\mathcal{C}\aleph_m\mathcal{P})^{\perp_1}\bigr)=\bigl(\mathcal{C}\aleph_mP\text{-}proj,\mathcal{C}\aleph_mP\text{-}inj\bigr),\]
that is cogenerated by the set $\mathcal{C}\aleph_m\mathcal{P}$. We call $\mathcal{C}$-$\aleph_mP\text{-}$projective any module in the class $\mathcal{C}\aleph_mP\text{-}proj$.
    
\end{Definition}

\begin{Remark} (i) For $m=0$ the above cotorsion pair is exactly the pair $\bigl(\mathsf{fpProj},\mathsf{fpInj}\bigr)$. 

(ii) For each $m$ it is easy to see that the inclusions
$\mathsf{fpProj}\subset \mathcal{C}\aleph_mP\text{-}proj\subset \mathsf{WfpProj}$
hold true. 

(iii) By (ii) we also have $\mathsf{SfpInj}\subset \mathcal{C}\aleph_mP\text{-}inj\subset \mathsf{fpInj}$.

(iv) \textbf{Notation:} Before the class of weakly FP-projective modules was named as such in \cite{Bazzoni2},  the symbol $\mathcal{C}$ was often used to denote the class $\mathsf{WfpProj}$. Here, we use $\mathcal{C}\aleph_mP\text{-}proj$ for the newly defined class since the following inclusion holds $\mathcal{C}\aleph_mP\text{-}proj\subset\mathsf{WfpProj}\bigcap \aleph_mP\text{-}proj$, where by $\aleph_mP\text{-}proj$ we mean the left class of the cotorsion pair that is cogenerated by the $\aleph_m$-presentable modules. Using the old notation, the above inclusion would be written as $\mathcal{C}\aleph_mP\text{-}proj\subset\mathcal{C}\bigcap( \aleph_mP\text{-}proj)$.
\end{Remark}

\begin{Lemma}\label{Filt} Let $m>0$ be such that $\aleph_m$ is a regular cardinal. Then a module $M$ is $\mathcal{C}$-$\aleph_mP\text{-}$projective if and only if it is filtered by $\aleph_m$-presentable weakly FP-projective modules. Equivalently, the equality $\operatorname{Filt}(\mathcal{C}\aleph_m\mathcal{P})=\mathcal{C}\aleph_mP\text{-}proj$ holds.
    
\end{Lemma}

\begin{proof}
The inclusion $\operatorname{Filt}(\mathcal{C}\aleph_m\mathcal{P})\subset \mathcal{C}\aleph_mP\text{-}proj$ is \cite[Corollary 6.14]{Approximations}.

For the inclusion $\mathcal{C}\aleph_mP\text{-}proj\subset \operatorname{Filt}(\mathcal{C}\aleph_m\mathcal{P})$ we just need to observe that for any $\aleph_m$-presentable $\mathcal{C}$-$\aleph_mP\text{-}$projective module $N$ a representative of $N$ belongs to the set $\mathcal{C}\aleph_m\mathcal{P}$ of Definition \ref{WNmP} and then use \cite[Theorem 7.13]{Approximations}. Indeed, let $N$ be an $\aleph_m$-presentable $\mathcal{C}$-$\aleph_mP\text{-}$projective module. Then, by the above remark, $N$ is an  $\aleph_m$-presentable weakly $FP\text{-}$projective module which means exactly that a representative of $N$ belongs to $\mathcal{C}\aleph_m\mathcal{P}$.
    
\end{proof}

The following Proposition appears as Proposition 4.3 in \cite{StovicekHillLemma}. We note that what we call here $\kappa$-presentable objects are called $<\kappa$-presentable objects in \cite{StovicekHillLemma}. The same is true for locally $\kappa$-presentable categories in this paper and locally $<\kappa$-presentable categories in \cite{StovicekHillLemma}. In addition, since we assume homological notation the complexes called bounded above here are called bounded below complexes in \cite{StovicekHillLemma}.

\begin{Proposition}[{\v{S}\v{t}ov\'{\i}\v{c}ek, Proposition~4.3}]
Assume that $\kappa$ is an infinite regular cardinal and that
$\mathcal{G}$ is a Grothendieck category which is locally
$\kappa$-presentable. Let $\mathcal{F}\subseteq\mathcal{G}$
be deconstructible. Choose a set $\mathcal{S}$ of all
representatives of $\kappa$-presentable
objects belonging to $\mathcal{F}$, and assume
\[
\mathcal{F}=\operatorname{Filt}\text{-}\mathcal{S}.
\]
Then every complex whose terms belong to $\mathcal{F}$ admits
a filtration with successive quotients bounded above complexes having all terms in $\mathcal{S}$. 
\end{Proposition}

For $\mathcal{G}=R\text{-}Mod$ the above proposition combined with Lemma \ref{Filt} immediately gives:

\begin{Corollary}\label{Filtabove} Let $m>0$ be such that $\aleph_m$ is a regular cardinal. Then any complex $X$ of $\mathcal{C}$-$\aleph_mP\text{-}$projective modules is filtered by bounded above complexes of $\aleph_m$-presentable weakly FP-projective modules.

\end{Corollary}

We are ready to state the main result of this section:

\begin{Theorem}\label{Main} Let $m>0$ be such that $\aleph_m$ is a regular cardinal. Assume that $X$ is a complex of $\mathcal{C}$-$\aleph_mP\text{-}$projective modules and that $E$ is an $\aleph_m$-pure acyclic complex of modules in the class $\mathcal{C}\aleph_mP\text{-}inj$. Then any chain map $X\rightarrow E$ is null-homotopic.
    
\end{Theorem}

\begin{proof} Since $\operatorname{Ext}^1_R(X_k,E_n)=0$ for every $k,\ n$,
we may use the following well-known isomorphism:
\[\operatorname{Hom}_{K(R)}(X,E)=\operatorname{Ext}^1_{Ch(R)}(X,\Sigma^{-1}E),\]
where by $\operatorname{Hom}_{K(R)}$ we denote the $\operatorname{Hom}$ functor of the homotopy category, $K(R),$ of the ring $R$. The statement means exactly that $\operatorname{Hom}_{K(R)}(X,E)=0$ so we only need to prove that $\operatorname{Ext}^1_{Ch(R)}(X,\Sigma^{-1}E)=0$. 

By Corollary \ref{Filtabove} $X$ is filtered by bounded above complexes $Y$ of $\aleph_m$-presentable weakly FP-projective modules. Hence, employing the Eklof lemma for $Ch(R)$ we only need to show that $\operatorname{Ext}^1_{Ch(R)}(Y,\Sigma^{-1}E)=0$, equivalently we only need to show that $\operatorname{Hom}_{K(R)}(Y,E)=0$. 

Since $Y$ is a complex of $\aleph_m$-presentable modules and $\aleph_m$-presentable modules are $\aleph_m$-pure projective the assumptions on $E$ and Theorem \ref{nullhomotopic} give $\operatorname{Hom}_{K(R)}(Y,E)=0$ as was required.

\end{proof}

\begin{Corollary}\label{wfpprojcycles} Let $R$ be any ring and $m\geq 0$ be an integer. Then every acyclic complex $X$ of $\mathcal{C}$-$\aleph_mP\text{-}$projective $R$-modules has weakly FP-projective cycles.
    
\end{Corollary}

\begin{proof} The case $m=0$ will be settled by the case $m>0$ since $\mathcal{C}$-$\aleph_0P\text{-}$projective modules are also $\mathcal{C}$-$\aleph_mP\text{-}$projective modules for any $m>0$.

Let $m>0$ be an integer, $S$ be a strongly FP-injective module, and let $E$ denote an
injective resolution of $S$.
 We also denote by $S^m$ the $m$th-cosyzygy of $S$ in
\[
E = 0 \longrightarrow S \longrightarrow I^0
  \longrightarrow I^1 \longrightarrow I^2
  \longrightarrow \cdots.
\]
and by $E^m$ the complex
\[
E^m =0 \longrightarrow S^m \longrightarrow I^m
  \longrightarrow I^{m+1} \longrightarrow \cdots.
\]

We only need to show that $\operatorname{Ext}^1(Z_n(X),S)=0$ for an arbitrary integer $n$. Since the modules of $X$ are weakly FP-projective a dimension shifting argument along $X$ followed by a dimension shifting argument along $E$ shows that
\[\operatorname{Ext}_R^1(Z_n(X),S)=\operatorname{Ext}_R^{m+1}(Z_{n-m}(X),S)=\operatorname{Ext}_R^1(Z_{n-m}(X),S^m).\]
Hence, we need to show that $\operatorname{Ext}_R^1(Z_{n-m}(X),S^m)=0$.

Consider the short exact sequence:
\[0\rightarrow Z_{n-m+1}(X)\rightarrow X_{n-m+1}\rightarrow Z_{n-m}(X)\rightarrow0,\] since $X_{n-m+1}$ is weakly FP-projective and $S^m$ is strongly FP-injective the abelian group $\operatorname{Ext}^1(X_{n-m+1},S^m)$ equals $0$. Therefore it is sufficient to show that the map \[i^*: \operatorname{Hom}_R(X_{n-m+1},S^m)\rightarrow \operatorname{Hom}_R(Z_{n-m+1}(X),S^m) \] is an epimorphism.

Consider the bounded above complex of injective modules 
\[I_\bullet^m=0\rightarrow I^m\rightarrow I^{m+1}\rightarrow\cdots,\] then since $X$ is acyclic it follows that $\operatorname{Hom}_{R}(X,I_\bullet^m)$ is an acyclic complex. By Example \ref{exampleofNmpureacyclic}, $E^m$ is an $\aleph_m$-pure acyclic complex so by Theorem \ref{Main} we have that $\operatorname{Hom}_{K(R)}(X,E^m)=0$. It is not difficult to see that we also have $\operatorname{Hom}_{K(R)}(X,\Sigma^kE^m)=0$ for every $k\in\mathbb{Z}$, hence $\operatorname{Hom}_R(X,E^m)$ is an acyclic complex. Finally, we just need to observe that the mapping cone of the obvious map $S^m\rightarrow I_\bullet^m$ is the complex $E^m$. Of course, now, we can deduce that the complex $\operatorname{Hom}_R(X,S^m)$ is acyclic and that $i^*$ is an epimorphism as needed.
    
\end{proof}

\begin{Remark} The special case of $m=0$ in the above Corollary is exactly \cite[Corollary 4.7]{Bazzoni2} which states that any acyclic complex of FP-projective modules has weakly FP-projective cycles. This is no surprise since the contents of this section are influenced by \cite{Bazzoni2}.

\end{Remark}

The following is a corollary of Corollary \ref{wfpprojcycles}.

\begin{Corollary}\label{Maincorollary} Let $R$ be a left $<\aleph_m$-coherent ring where $m$ is a non-negative integer. Then acyclic complexes of weakly FP-projective modules have weakly FP-projective cycles.
    
\end{Corollary}

\begin{proof} Over an $<\aleph_m$-coherent ring the modules in the cogenerating set $\mathcal{T}$ of Proposition \ref{module-pair} can be chosen to be $\aleph_m$-presentable. Hence, by \cite[Theorem 7.13]{Approximations} $\mathcal{T}$  and the set $\mathcal{C}\aleph_m\mathcal{P}$ of Definition \ref{WNmP} cogenerate the same cotorsion pair for any $m>0$ while for $m=0$ these sets are sets of representatives of all finitely presentable modules. This means exactly that, over $<\aleph_m$-coherent rings weakly FP-projective modules coincide with $\mathcal{C}$-$\aleph_mP$-projective modules. The desired result follows from Corollary \ref{wfpprojcycles}.
    
\end{proof}

\begin{Remark} The special case for coherent rings appeared first in \cite[Example 4.3]{SS2020}.  
    
\end{Remark}

The following definition follows the language introduced in \cite{Bazzoni}.

\begin{Definition} Let $\mathcal{M}$ be a class of modules. A module $N$ is called $\mathcal{M}$-periodic if it fits in a short exact sequence of the form:
\[0\rightarrow N\rightarrow M\rightarrow N\rightarrow 0,\]
where $M$ is a module from $\mathcal{M}$.
    
\end{Definition}

The following lemma is well known.

\begin{Lemma}\label{per-acy} Let $\mathcal{M}$ be a class of modules closed under countable direct sums and direct summands. Then the following conditions are equivalent:

(i) Every $\mathcal{M}$-periodic module belongs to $\mathcal{M}$.

(ii) Every cycle of an acyclic complex with entries in $\mathcal{M}$ belongs to $\mathcal{M}$. 
    
\end{Lemma}

The following proposition is obvious.

\begin{Proposition}
    
 Let $m$ be a non-integer integer and $N$ be a $\mathcal{C}\aleph_mP\text{-}proj$-periodic module. Then $N$ is a weakly FP-projective module.
    
\end{Proposition}

The preceding proposition also has a Cartan--Eilenberg version. Here CE-exactness allows
the result to be applied simultaneously to components,
cycles, boundaries, and homology.

\begin{Corollary}\label{ce-wfp-periodicity}
Let $R$ be any ring and let $m\geq0$ be an integer.
Suppose that
\[
    0\longrightarrow X\longrightarrow P
      \longrightarrow X\longrightarrow0
\]
is a CE-exact sequence of complexes with
\[
    P\in\operatorname{CE}
       \bigl(\mathcal C\aleph_mP\text{-}proj\bigr).
\]
Then
\[
    X\in\operatorname{CE}
       \bigl(\mathsf{WfpProj}(R)\bigr).
\]

In particular, if $R$ is left $<\aleph_m$-coherent,
then every complex that is periodic with respect to
$\operatorname{CE}(\mathsf{WfpProj}(R))$ in the
CE-exact structure belongs to
$\operatorname{CE}(\mathsf{WfpProj}(R))$.
\end{Corollary}

\begin{proof}
For every $n\in\mathbb Z$, applying any of the functors
\[
    (-)_n,\qquad Z_n,\qquad B_n,\qquad H_n
\]
gives a short exact sequence
\[
    0\longrightarrow M\longrightarrow Q
      \longrightarrow M\longrightarrow0
\]
with $Q\in\mathcal C\aleph_mP\text{-}proj$.
Splicing copies of this sequence gives an acyclic
complex with all components equal to $Q$ and all
cycles isomorphic to $M$.
Corollary~\ref{wfpprojcycles} shows that $M$ is
weakly FP-projective.

Thus the components, cycles, boundaries, and homology
modules of $X$ are weakly FP-projective, as required.
For the final assertion, use
$\mathcal C\aleph_mP\text{-}proj=\mathsf{WfpProj}(R)$
over a left $<\aleph_m$-coherent ring.
\end{proof}

We will show that it is not true, generally, that the cycles of acyclic complexes of $\mathcal{C}$-$\aleph_m$P-projective modules are $\mathcal{C}$-$\aleph_{m+1}$P-projective modules. But first, we establish the following:

\begin{Proposition}\label{lastprop} Let $R$ be a ring and $m$ be such that $\aleph_m$ is a regular cardinal. Consider the following conditions:

(i) $R$ is $<\aleph_m$-coherent.

(ii) The cotorsion pairs $\bigl(\mathcal{C}\aleph_mP\text{-}proj,\mathcal{C}\aleph_mP\text{-}inj\bigr)$ and $\bigl(\mathsf{WfpProj}(R),\mathsf{SfpInj}(R)\bigr)$ are equal.

(iii) The cotorsion pair $\bigl(\mathcal{C}\aleph_mP\text{-}proj,\mathcal{C}\aleph_mP\text{-}inj\bigr)$ is hereditary.

(iv) Every acyclic complex of $\mathcal{C}$-$\aleph_m$P-projective modules has $\mathcal{C}$-$\aleph_m$P-projective cycles.  

Then (i),(ii) and (iii) are equivalent. Furthermore, (iv) implies (i),(ii) and (iii). If $m$ is also an integer then all four conditions are equivalent.
\end{Proposition}

\begin{proof}

Even though we can directly prove the case $m=0$ we only prove the statement for $m>0$ such that $\aleph_m$ is a regular cardinal. For the case $m=0$, we refer the reader to \cite[Corollary 4.6]{Bazzoni2}.

The implication (i)$\Rightarrow$(ii) has been explained in the proof of Corollary \ref{Maincorollary}. The implication (ii)$\Rightarrow$(iii) is obvious. 
We now prove the implication $(iii)\Rightarrow(i)$. Employing Lemma \ref{l-coherent} we only need to show that any finitely generated left ideal $I$ of $R$ is $\aleph_m$-presentable. $R/I$ is a finitely presented module so by assumption (iii) we deduce that $I$ is a $\mathcal{C}$-$\aleph_m$P-projective module. By the definition of the set $\mathcal{C}\aleph_m\mathcal{P}$ and the Kaplansky Theorem for cotorsion pairs \emph{(\cite[Theorem 7.13]{Approximations})} we deduce that $I$ is filtered by $\aleph_m$-presentable modules. Finally, $I$ is $\aleph_m$-presentable by \cite[Theorem 7.10(Hill Lemma ($H_4$))]{Approximations}, use there $M=I$, $N=0$ and $X$ any finite generating set of $I$.

For the implication (iv)$\Rightarrow$(iii) we only need to show that $\mathcal{C}$-$\aleph_m$P-projective modules are closed under kernels of epimorphisms. Assume that \[0\rightarrow K\rightarrow A\rightarrow B\rightarrow 0\] is a short exact sequence with $A,B\in\mathcal{C}\aleph_mP\text{-}proj$. We need to show that $K\in\mathcal{C}\aleph_mP\text{-}proj$. Splicing the above exact sequence with a projective resolution of $K$ we get an acyclic complex $P$ of $\mathcal{C}$-$\aleph_m$P-projective modules. Since $K$ appears as a cycle of $P$, the assumption (iv) gives $K\in\mathcal{C}\aleph_mP\text{-}proj$.

For an integer $m>0$, the implication (i)$\Rightarrow$(iv) is exactly Corollary \ref{Maincorollary}.

\end{proof}

\begin{Corollary}\label{cross-cardinal-periodicity}
Let $R$ be a ring, and let $m<n$ be ordinals such that
$\aleph_m$ and $\aleph_n$ are regular cardinals.
Consider the following conditions:
\begin{enumerate}[label=(\roman*)]
\item $R$ is left $<\aleph_n$-coherent.
\item Every acyclic complex of
$\mathcal C\aleph_mP$-projective modules has
$\mathcal C\aleph_nP$-projective cycles.
\item Every
$\mathcal C\aleph_mP\text{-}proj$-periodic module
is $\mathcal C\aleph_nP$-projective.
\end{enumerate}
Then {\rm(ii)} and {\rm(iii)} are equivalent,
and each implies {\rm(i)}.
If $m$ is a non-negative integer, all three
conditions are equivalent.
\end{Corollary}

\begin{proof}

The equivalence of (ii) and (iii) follows by a similar argument as in Lemma \ref{per-acy}. Assuming (ii), since projective modules and finitely presented modules are $\mathcal{C}$-$\aleph_m$P-projective modules, by the assumption we can deduce that all modules in the set $\mathcal{T}$ of Proposition \ref{module-pair} are $\mathcal{C}$-$\aleph_{n}$P-projective modules. Hence, $\mathcal{C}$-$\aleph_{n}$P-projective modules and weakly FP-projective modules coincide, so by Proposition \ref{lastprop} $R$ is  $<\aleph_n$-coherent.

 Finally, suppose $m$ is a non-negative integer
and (i) holds. Proposition~\ref{lastprop} gives
$\mathcal C\aleph_nP\text{-}proj=\mathsf{WfpProj}(R)$.
By Corollary~\ref{wfpprojcycles}, every acyclic complex
with components in $\mathcal C\aleph_mP\text{-}proj$ has weakly
FP-projective cycles, proving (ii).
\end{proof}

\noindent\textbf{Funding.} I. Kaperonis and K.~Golfis were supported by the Hellenic Foundation for Research and Innovation (H.F.R.I.) under the ``3rd Call for H.F.R.I. Research Projects to Support Faculty Members and Researchers'', project number 24921.

\bigskip
\noindent\begin{minipage}{\textwidth}
\small
Ilias Kaperonis\\
Department of Mathematics, University of Athens, Athens 15784, Greece\\
E-mail: \href{mailto:kaperonn@math.uoa.gr}{\texttt{kaperonn@math.uoa.gr}}\\
Konstantinos Golfis\\
Department of Mathematics, University of Athens, Athens 15784, Greece\\
E-mail: \href{mailto:golfisk@math.uoa.gr}{\texttt{golfisk@math.uoa.gr}}
\end{minipage}

\end{document}